\documentclass{article}
\usepackage[english]{babel}
\usepackage[T1]{fontenc}
 \usepackage{charter,amsmath,amssymb,amsthm,mathrsfs,tikz-cd} 
\usepackage[backend=bibtex, style=alphabetic, maxbibnames=99]{biblatex}
\DeclareFieldFormat{postnote}{\mknormrange{#1}}
\DeclareFieldFormat{volcitepages}{\mknormrange{#1}}
\DeclareFieldFormat{multipostnote}{\mknormrange{#1}}
\usepackage[hidelinks]{hyperref}
\usepackage{orcidlink}\title{Normality of monodromy group in generic convolution group}
\date{\today}
\author{Haohao \textsc{Liu}\thanks{email: \href{mailto:kyung@mail.ustc.edu.cn}{kyung@mail.ustc.edu.cn}, IRMA, Université de Strasbourg, 7 rue René-Descartes, 67084 Strasbourg Cedex, France}\,\orcidlink{0009-0007-5942-8174}}

\usepackage{CJKutf8}%

\newcounter{relctr} %
\everydisplay\expandafter{\the\everydisplay\setcounter{relctr}{0}} %

\newcommand\labelrel[2]{%
	\begingroup
	\refstepcounter{relctr}%
	\stackrel{\textnormal{(\alph{relctr})}}{\mathstrut{#1}}%
	\originallabel{#2}%
	\endgroup
}
\AtBeginDocument{\let\originallabel\label} %

\makeatletter
\newcommand\setItemnumber[1]{\setcounter{enum\romannumeral\@enumdepth}{\numexpr#1-1\relax}}
\makeatother

\newtheorem{thm}{Theorem}[subsection]
\newtheorem*{thm*}{Theorem}
\newtheorem{lm}[thm]{Lemma}

\newtheorem*{cor*}{Corollary}

\newtheorem{ft}[thm]{Fact}
\theoremstyle{remark}
\newtheorem{rk}[thm]{Remark}
\newtheorem{claim}[thm]{Claim}
\theoremstyle{definition}
\newtheorem{df}[thm]{Definition}
\newtheorem{eg}[thm]{Example}
\newtheorem{set}[thm]{Setting}
\def\<{\langle}
\def\>{\rangle}
\def\ab{\mathrm{ab}}
\def\an{\mathrm{an}}
\DeclareMathOperator{\Aut}{Aut}
\def\bV{\mathbb{V}}
\def\C{\mathbb{C}}
\def\cA{\mathcal{A}}
\def\cB{\mathcal{B}}
\def\cC{\mathcal{C}}

\def\cE{\mathcal{E}}

\def\cH{\mathcal{H}}

\def\cM{\mathcal{M}}
\DeclareMathOperator\Cons{Cons}
\def\cP{\mathcal{P}}
\def\cQ{\mathcal{Q}}
\def\cR{\mathcal{R}}
\def\cS{\mathcal{S}}

\def\D{\mathbb{D}}

\def\G{\mathbb{G}}
\def\GL{\mathrm{GL}}

\def\Id{\mathrm{Id}}

\def\iso{\xrightarrow{\sim}}

\DeclareMathOperator\Loc{Loc}
\DeclareMathOperator\Mon{Mon}
\DeclareMathOperator\MT{MT}
\def\N{\mathbb{N}}
\def\op{\mathrm{op}}
\def\Perv{\mathrm{Perv}}
\def\pet{\pi_1}
\def\Q{\mathbb{Q}}

\def\red{\mathrm{red}}
\DeclareMathOperator{\Rep}{Rep}
\DeclareMathOperator{\RHom}{R\cH om}

\def\sh{\mathrm{sh}}

\DeclareMathOperator{\Spec}{Spec}
\DeclareMathOperator{\Supp}{Supp}

\def\ULA{\mathrm{ULA}}
\def\Vec{\mathrm{Vec}}
\def\Z{\mathbb{Z}}
\begin{document}
	\maketitle
	\begin{abstract}On an abelian variety $A$, sheaf convolution  gives a Tannakian formalism for perverse sheaves. Let $X$ be an irreducible algebraic variety with generic point $\eta$. Let $K$ be a family of perverse sheaves (more precisely, a  relative perverse sheaf) on the constant abelian scheme $p_X:A\times X\to X$. We show that for uncountably many character sheaves $L_{\chi}$ on $A$, the monodromy groups of  $R^0p_{X*}(K\otimes p_A^*L_{\chi})$ are normal in the Tannakian group $G(K|_{A_{\eta}})$ of the perverse sheaf $K|_{A_{\eta}}\in\Perv(A_{\eta})$. 
		
		This result is inspired from and could be compared to two other normality results: In the same setting,  the Tannakian group $G(K|_{A_{\bar{\eta}}})$ is  normal  in $G(K|_{A_{\eta}})$ (due to Lawrence-Sawin). For a polarizable variation of Hodge structures, outside a meager locus, the connected monodromy group is normal in the derived Mumford-Tate group (due to André).
	\end{abstract}
	
	\section{Introduction}
\subsection{Background}
Constructing  local systems (or $\ell$-adic lisse sheaves)  with a prescribed monodromy group is an important problem having a long history.

 In positive characteristics, Katz and his collaborators exhibit  local systems whose monodromy groups are the simple algebraic group  $G_2$ (\cite[11.8]{katz88gauss}), $2.J_2$ (\cite{katz2019J2}), the finite symplectic groups (\cite{katz2019rigid}), the special unitary groups (\cite{katz2019local}), \textit{etc}. In particular, the exceptional Lie groups  appear unexpectedly in algebraic geometry.

In characteristic zero, such constructions help to understand Galois groups of number fields. Dettweiler and  Reiter \cite{dettweiler2010rigid} prove the existence of a local system on $\mathbf{P}^1_{\Q}\setminus\{0,1,\infty\}$ whose monodromy group is $G_2$. It  produces a motivic Galois representations with image dense in $G_2$. Their proof relies on  Katz's middle convolution of perverse sheaves. Yun \cite{yun2014motives}  constructs local systems with some other exceptional groups as monodromy groups. As applications, he answers a long standing question of Serre, and solves new cases of the inverse Galois problem.  His construction uses the geometric Langlands correspondence.
 
A new proof of Mordell's conjecture \cite{lawrence2020diophantine}, and its subsequent generalization to higher dimensional varieties over number fields, rely on the existence of local systems with big monodromy  over the variety in question.  Lawrence and Sawin \cite{lawrence2020shafarevich} use this  technique  to prove Shafarevich's conjecture for hypersurfaces in abelian varieties. Krämer and Maculan \cite{kramer2023arithmetic} apply roughly the same strategy to obtain an arithmetic finiteness result for very irregular varieties of dimension less than half the dimension of their Albanese variety. In both cases, the construction of   local systems  uses perverse sheaves. \subsection{Convolution groups and monodromy groups}\label{sec:monoVSconv}
In \cite{lawrence2020shafarevich}, the construction of local systems rests on comparing the monodromy group with the Tannakian group from Krämer-Weissauer's  convolution theory \cite{kramer2015vanish}.   As \cite[p.4]{javanpeykar2023monodromy} comments, this comparison is similar to the study of monodromy groups \emph{via} Mumford-Tate groups in \cite{andre1992mumford}.   

We briefly outline the argument of \cite{lawrence2020shafarevich}. On an abelian variety $A$,  a quotient of the abelian category $\Perv(A)$ of perverse sheaves on $A$ is a Tannakian category under sheaf convolution. Let $X$ be a normal irreducible algebraic variety with generic point $\eta$. Let $K$ be a universally locally acyclic, relative perverse sheaf on the constant abelian scheme $p_X:A\times X\to X$. Intuitively, $K\in D_c^b(A\times X)$ is a family of perverse shaves on $A$ parameterized by $X$. For a character sheaf $L_{\chi}$ on $A$,  $L^0(K,\chi):=R^0p_{X*}(K\otimes p_A^*L_{\chi})$ is a lisse sheaf on $X$ (Remark \ref{rk:monogp}). Let $\rho_{K,\chi}:\pi_1(X,\bar{\eta})\to \GL\left(H^0(A_{\bar{\eta}},K|_{A_{\eta}}\otimes L_{\chi})\right)$ be the corresponding monodromy representation. Let $\Mon(K,\chi)$ be the Zariski closure of the image of $\rho_{K,\chi}$.

 By \cite[Lemma~2.8]{lawrence2020shafarevich} (see also \cite[Thm.~4.3]{javanpeykar2023monodromy}), the Tannakian group $G(K|_{A_{\bar{\eta}}})$ of $K|_{A_{\bar{\eta}}}\in \Perv(A_{\bar{\eta}})$ is  normal in the  Tannakian group $G(K|_{A_{\eta}})$ of $K|_{A_{\eta}}\in \Perv(A_{\eta})$. This normality is used in \cite[Theorem 4.7]{lawrence2020shafarevich} to prove that, roughly speaking, if $K$ is associated with a family of hypersurfaces, then for most character sheaves $L{_\chi}$ on $A$, the  monodromy groups $\Mon(K,\chi)$ contain $G(K|_{A_{\bar{\eta}}})$.   For such $L_{\chi}$, one can apply  Lawrence-Venkatesh's machinery to the lisse sheaves $L^0(K,\chi)$.
\subsection{Statements}
Our main result bears formal analogy with  André's normality theorem, which we recall. The category of rational Hodge structures is Tannakian. For a rational Hodge structure $V$, the Tannakian group of the  Tannakian subcategory $\<V\>$ generated by $V$ is called the \emph{Mumford-Tate group} of $V$ and denoted by $\MT(V)$. Let $X$ be a complex smooth quasi-projective variety. Let $\bV$ be a  polarizable variation of rational Hodge structure on $X(\C)$. For every $x\in X(\C)$, let $\rho:\pi_1(X(\C),x)\to \GL(\bV_x)$ be  the monodromy representation.
The identity component $H_x$ of the Zariski closure $\overline{\mathrm{Im}(\rho)}^{\mathrm{Zar}}$ of the image of $\rho$ is called the \emph{connected monodromy group}.  \begin{ft}[{\cite[Thm.~1]{andre1992mumford}, see also \cite[Prop.~15.3.9]{carlson2017period}}]\label{ft:Andrenormal}There is a countable union $Z$ of strict Zariski closed subsets of $X$, such that for every $x\in X(\C)\setminus Z$,  $H_x$ is a normal subgroup of  $\MT(\bV_x)$.
\end{ft}
In the main result, Theorem \ref{thm:mainnormal}, we prove that for a semisimple, relative perverse sheaf $K$, its generic convolution group $G(K|_{A_{\eta}})$  is reductive. Moreover, for many characters $\chi:\pi_1(A)\to \bar{\Q}_{\ell}^{\times}$, the monodromy group $\Mon(K,\chi)$ is a normal subgroup of $G(K|_{A_{\eta}})$. This normality puts further restriction on the monodromy group. Using Krämer's method \cite[Thm.~6.2.1]{kramer2022characteristic1}, Lawrence and Sawin \cite[Lemma 3.9]{lawrence2020shafarevich} even show that the geometric generic convolution group $G(K|_{A_{\bar\eta}})$ modulo center is simple.

\begin{set}\label{set:AX}Let $k$ be an algebraically closed field of characteristic zero. Let $X$ be an integral algebraic variety over $k$ with generic point $\eta$. Let $A$ be an abelian variety over $k$. Denote by $p_X:A\times X\to X$ and $p_A:A\times X\to A$ the projections.\end{set} Let $\ell$ be a prime number. Let $\bar{\Q}_{\ell}$ be an algebraic closure of $\Q_{\ell}$.  Let $D_c^b(A\times X)$ be the triangulated category  of bounded constructible $\bar{\Q}_{\ell}$-sheaves on $A\times X$. 
Let $\pet(A)$ be the étale fundamental group of $A$ based at the geometric origin point.    Fix a relative perverse sheaves  $K$ for $p_X:A\times X\to X$ in the sense of Definition \ref{df:relperv}. Assume  that $K$ is a \emph{semisimple} object of $D_c^b(A\times X)$ (in the sense of Definition \ref{df:ssDcb}).  Let $G_{\omega_{\chi}}(K|_{A_{\eta}})$ be the Tannakian monodromy group (Definition \ref{df:Tannakianmonodromy}) of $K|_{A_{\eta}}$, referred to as the \emph{generic convolution group}.

\begin{thm}[Lemma \ref{lm:monreductive}, Theorem \ref{thm:normal}]\label{thm:mainnormal}
Assume $\dim A>0$. Then there are uncountably many characters $\chi:\pet(A)\to \bar{\Q}_{\ell}^{\times}$, such that  $G_{\omega_{\chi}}(K|_{A_{\eta}})$ is a well-defined reductive group. It contains  $\Mon(K,\chi)$ as a  closed, reductive, \emph{normal} subgroup.
\end{thm}\begin{rk}\begin{enumerate}\item In fact, Gabber and Loeser \cite{gabber1996faisceaux} introduce a scheme  $\cC(A)_{\ell}$ (reviewed in Definition \ref{df:cotorus}) to parameterize  pro-$\ell$  characters of $\pet(A)$. As in Fact \ref{ft:Andrenormal}, what we exclude in Theorem \ref{thm:mainnormal} is also a countable union of strict Zariski closed subsets of this scheme.
	\item By \cite[Cor.~13]{peters2003monodromy}, the category of polarizable variation of rational Hodge structure on $X(\C)$ in Fact \ref{ft:Andrenormal} is semisimple. In this sense, the semisimplicity hypothesis in Theorem \ref{thm:mainnormal} should be compared to the polarizability assumption in Fact \ref{ft:Andrenormal}. As Remark \ref{rk:fromsubvar} explains, if $K$ is associated with a family of closed subvarieties of $A$, then it is semisimple in $D_c^b(A\times X)$.
\item In Fact \ref{ft:Andrenormal},   the connected monodromy group $H_x$ is independent of the choice of $x$ up to isomorphism. By contrast, in Theorem \ref{thm:mainnormal}, the Tannakian group $G_{\omega_{\chi}}(K|_{A_{\eta}})$ is independent of the choice of character $\chi$ up to isomorphism. \end{enumerate}  \end{rk}
The line of the proof of Theorem \ref{thm:mainnormal} is  similar to that of  Fact \ref{ft:Andrenormal}.  As André \cite[p.10]{andre1992mumford} explains, the normality  is a consequence of the theorem of the fixed part due to Griffiths-Schmidt-Steenbrink-Zucker.  In our case, an analog of the theorem of the fixed part is Theorem \ref{thm:introfixed}.

 For a character $\chi_{\ell'}$ of $\pi_1(A)$ of finite order prime to $\ell$, and a pro-$\ell$ character $\chi_{\ell}\in \cC(A)_{\ell}$,  set $\chi=\chi_{\ell'}\chi_{\ell}$. Let $\Perv^{\ULA}(A\times X/X)\subset D_c^b(A\times X)$ be the  full subcategory  of $p_X$-universally locally acyclic (ULA, reviewed in Definition \ref{df:ULA}) relative perverse sheaves. Then $\Perv^{\ULA}(A\times X/X)$ is an abelian category. 
\begin{thm}[Theorem \ref{thm:fixedpart}]\label{thm:introfixed} Assume that $X$ is smooth. Let $K\in \Perv^{\ULA}(A\times X/X)$ be semisimple in $D_c^b(A\times X)$. Then there is a subobject $K^0\subset K$ in $\Perv^{\ULA}(A\times X/X)$ with the following property: For every character $\chi_{\ell'}:\pet(A)\to\bar{\Q}_{\ell}^{\times}$ of finite order prime to $\ell$, there is a nonempty Zariski open subset $U\subset \cC(A)_{\ell}$,  such that for every $\chi_{\ell}\in U$, one has \[H^0(A_{\bar{\eta}},K^0|_{A_{\bar{\eta}}}\otimes^LL_{\chi})=H^0(A_{\bar{\eta}},K|_{A_{\bar{\eta}}}\otimes^LL_{\chi})^{\Gamma_{k(\eta)}}.\]
\end{thm}
The proof of Theorem \ref{thm:introfixed} uses the projection $p_A:A\times X\to A$, which restricts our results to constant abelian schemes. We leave the question  whether Theorem \ref{thm:introfixed} has an analog for relative perverse sheaves on an arbitrary (non-constant) abelian scheme.
\subsection*{Notation and conventions}
 An object of an abelian category is  \emph{semisimple} if it is the direct sum of finitely many simple objects. An abelian category is  \emph{semisimple} if every object is semisimple. For a field $k$, its absolute Galois group is denoted by $\Gamma_k$.  An algebraic variety means a scheme of finite type and separated over $k$. A linear algebraic group is   \emph{reductive}, if its identity component is  reductive (in the  sense of \cite[6.46, p.135]{milne2017algebraic}). For a topological group, its $\bar{\Q}_{\ell}$-characters are assumed to be continuous. For an irreducible algebraic variety $X$ (on which $\ell$ is invertible) and a character $\chi:\pet(X)\to \bar{\Q}_{\ell}^{\times}$ of its étale fundamental group $\pet(X)$, let $L_{\chi}$  be the corresponding rank one $\bar{\Q}_{\ell}$- lisse sheaf on $X$.

	\section{Recollections on constructible sheaves}\label{sec:perv}
No originality is claimed in Section \ref{sec:perv}. Let $k$ be a field. Let $\ell$ be a prime number invertible in $k$. Fix an algebraic closure $\bar{k}$ of $k$. For every algebraic variety $X$ over $k$, denote by $D_c^b(X):=D_c^b(X,\bar{\Q}_{\ell})$  the triangulated category of complexes of $\bar{\Q}_{\ell}$-sheaves on $X$ with bounded constructible cohomologies defined in \cite[p.74]{beilinson2018faisceaux}. Let $\D_X:D_c^b(X)\to D_c^b(X)^{\op}$ be the Verdier duality functor.  The heart of the standard  t-structure on $D_c^b(X)$ is denoted by $\Cons(X)$, which is the category of constructible $\bar{\Q}_{\ell}$-sheaves on $X$. For every integer $n$, let $\cH^n:D_c^b(X)\to\Cons(X)$ be the functor taking the $n$-th  cohomology sheaf.   For   $F\in \Cons(X)$, set $\Supp F:=\{x\in X|F_x\neq0\}$ to be its support. Then $\Supp F$ is a quasi-constructible subset of $X$ in the sense of \cite[10.1.1]{EGAIV3}.  Let $\Loc(X)\subset\Cons(X)$ be the full subcategory of  $\bar{\Q}_{\ell}$-lisse sheaves on $X$. 

For every subset $S\subset X$, let $\bar{S}$ be its Zariski closure. Let ${}^pD^{\le0}(X)\subset D_c^b(X) $ be the full subcategory of objects $K$ with $\dim \overline{\Supp \cH^nK}\le -n$ for all integers $n$. Let  ${}^pD^{\ge0}(X)\subset D_c^b(X) $ be the full subcategory of objects $K$ with $\D_XK\in {}^pD^{\le0}(X)$. Then $({}^pD^{\le0}(X),{}^pD^{\ge0}(X))$ defines the (absolute) perverse t-structure on $D_c^b(X)$, whose heart   $\Perv(X)$ is the category of perverse sheaves on $X$. The functor $\D_X$ interchanges  ${}^pD^{\le0}(X)$ and ${}^pD^{\ge0}(X)$.    For every integer $n$, let  ${}^p\cH^n:D_c^b(X)\to \Perv(X)$ be the functor taking the $n$-th perverse cohomology sheaf.
For a morphism $f:X'\to X$ of schemes and $K\in D_c^b(X)$, set $K|_{X'}:=f^*K$.  
\subsection{Basics}
\begin{ft}[Projection formula, {\cite[Rk. (2), p.100]{freitag2013etale}, \cite[\href{https://stacks.math.columbia.edu/tag/0F10}{Tag 0F10 (1)}]{stacks-project}}]\label{ft:cstprojection}
	Let $f:X\to Y$ be a morphism of algebraic varieties over $\bar{k}$. Let $L\in D_c^b(Y)$ be an object with $\cH^nL\in \Loc(X)$ for all integers $n$. Then there is a natural isomorphism $(Rf_*-)\otimes^LL\to Rf_*(-\otimes^Lf^*L)$  of functors $D_c^b(X)\to D_c^b(Y)$.
\end{ft}

\begin{ft}[{\cite[Prop.~12.10]{freitag2013etale}}]\label{ft:genericlisse}Let $X$ be an algebraic variety over $k$. For every $F\in \Cons(X)$, there is a nonempty Zariski open subset $U\subset X$ with $F|_U\in \Loc(U)$. \end{ft}

	\begin{df}[{\cite[Def.~78]{brosnan2018unit}}]\label{df:ssDcb}\index{semisimple constructible complex of sheaves}
Let $X$ be an algebraic variety over $k$.	An object $K\in D_c^b(X)$ is called \emph{semisimple} if it is isomorphic to a finite direct sum of degree shifts of  semisimple objects of $\Perv(X)$.
\end{df}
If $K\in D_c^b(X)$ is semisimple,	then it is isomorphic to $\oplus_{n\in \Z}{}^p\cH^n(K)[-n]$ in $D_c^b(X)$, and each ${}^p\cH^n(K)$ is a semisimple object of $\Perv(X)$. A degree shift of a semisimple object of $D_c^b(X)$ is still semisimple. 

Lemma \ref{lm:shriksimple} is used in the proof of Lemma \ref{lm:monreductive}.
\begin{lm}\label{lm:shriksimple}Let $X$ be an algebraic variety over $k$.
Let $U\subset X$ be an open subset of $X$. Then the  functor $(-)|_U:\Perv(X)\to\Perv(U)$ sends every simple object of $\Perv(X)$ to a simple or zero object of $\Perv(U)$. In particular, the functor $(-)|_U:D_c^b(X)\to D_c^b(U)$ preserves semisimplicity.
\end{lm}
\begin{proof}
Let $K$ be a simple object of $\Perv(X)$. By \cite[Thm.~4.3.1 (ii)]{beilinson2018faisceaux}, there is an irreducible, locally closed and geometrically smooth subvariety $j:V\to X$ and a simple $\bar{\Q}_{\ell}$-lisse sheaf $L$ on $V$, such that $K$ is isomorphic to $j_{!*}L[\dim V]$. If $V$ is disjoint from $U$, then $K|_U=0$. Now assume that $V$ intersects $U$. Take a geometric point $\bar{x}$ on $V\cap U$. From \cite[V, Prop.~8.2]{SGA1},  as $V$ is normal, the morphism $\pet(U\cap V,\bar{x})\to \pet(V,\bar{x})$ is surjective. Thus, the composite representation $\pet(U\cap V,\bar{x})\to \GL(L_{\bar{x}})$ is also simple, \textit{i.e.}, the $\bar{\Q}_{\ell}$-lisse sheaf $L|_{U\cap V}$ is simple. Let $h:U\cap V\to U$ be the immersion. Then $K|_U$ is isomorphic to $h_{!*}L|_{U\cap V}[\dim (U\cap V)]$, hence simple in $\Perv(U)$.
\end{proof}
When $k=\C$,	Fact \ref{ft:BBD} \ref{it:decomp} follows from Kashiwara's conjecture for semisimple perverse sheaves and the paragraph following \cite[Thm.~6.2.5]{beilinson2018faisceaux}. Kashiwara's conjecture is formulated in \cite[Sec.~1]{kashiwara1998semisimple}; see also \cite[Sec.~1.2, 1]{drinfeld2001conjecture}. It is reduced to de Jong's conjecture by Drinfeld \cite{drinfeld2001conjecture}, which in turn is proved in \cite{bockle2006mod} and \cite{gaitsgory2007jong}. The  case of general $k$ follows \textit{via} Lemma \ref{lm:pervext}; see also \cite[Sec.~1.7]{drinfeld2001conjecture}.
\begin{ft}\label{ft:BBD}
	Let $k$ be an algebraically closed field of characteristic $0$. Let $f:X\to Y$ be a proper morphism of algebraic varieties over $k$. Let $K$ be a semisimple object of $D_c^b(X)$.  
	\begin{enumerate}\item\label{it:decomp} \textup{(Decomposition theorem)} Then  $Rf_*K$ is a semisimple object of $D_c^b(Y)$.		\item\label{it:globalinv} \textup{(Global invariant cycle theorem, \cite[Cor.~6.2.8]{beilinson2018faisceaux})} Let $i$ be an integer. By Fact \ref{ft:genericlisse}, there is  a nonempty connected open subset $V\subset Y$  such that $\cH^iRf_*K|_V$ is a lisse sheaf. Then for every  $y\in V(k)$, the canonical map \[H^i(X,K)\to H^i(X_y,K|_{X_y})^{\pet(V,y)}\] is surjective.\end{enumerate}
\end{ft}\begin{lm}\label{lm:pervext}
Let $E/F$ be an extension of algebraically closed fields. Let $X$ be an algebraic variety over $F$. Then:
\begin{enumerate} \item\label{it:JKLMA1} The functor $(-)|_{X_E}:D_c^b(X)\to D_c^b(X_E)$ is  fully faithful. It restricts to an exact functor $\Perv(X)\to \Perv(X_E)$.
	\item\label{it:scalarextensionofsimple} An object of $\Perv(X)$ is simple (resp. semisimple) if and only if its image under $(-)|_{X_E}:\Perv(X)\to \Perv(X_E)$ is simple (resp. semisimple). \end{enumerate}
\end{lm}
\begin{proof}
\begin{enumerate}
\item In characteristic zero, it is  the  first half of \cite[Lem.~A.1]{javanpeykar2023monodromy}. That  proof works in positive characteristic as well. 
\item Let $K\in\Perv(X)$. By \cite[Thm.~4.3.1 (ii)]{beilinson2018faisceaux} and \cite[Prop.~5.3]{esnault2017survey},  $K$ is simple if and only if   $K|_{X_E}$ is simple. Thus, if $K$ is semisimple, then so is $K|_{X_E}$. Conversely, assume that $K|_{X_E}$ is semisimple. For every subobject $P\subset K$ in $\Perv(X)$, there is a morphism $r:K|_{X_E}\to P|_{X_E}$ in $\Perv(X_E)$ with $r|_{(P|_{X_E})}=\Id_{P|_{X_E}}$. By Part \ref{it:JKLMA1}, there is a morphism $r':K\to P$ in $\Perv(X)$ with $r'|_{(K|_{X_E})}=r$ and $r'|_P=\Id_P$. Thus, $P$ admits a direct complement in $K$. By Lemma \ref{lm:finisubq} \ref{it:subhascom}, $K$ is semisimple. \end{enumerate}
\end{proof}\begin{lm}\label{lm:finisubq}Let $\cA$ be an abelian category. Let $X\in \cA$ be a Noetherian and Artinian object. 
\begin{enumerate}\item\label{it:subq=factor} Let $Y$ be a simple subquotient of $X$. Then there is a composite series of $X$ with one graded piece isomorphic to $Y$. In particular, up to isomorphism $X$ has only finitely many simple subquotients.
	\item\label{it:subhascom} If every subobject of $X$ admits a direct complement, then $X$ is semisimple.
\end{enumerate}\end{lm}
\begin{proof}\hfill\begin{enumerate}\item There is a subobject $i:X_0\subset X$ and a quotient $q:X_0\to Y$ in $\cA$. Let $N=\ker(q)$. By \cite[\href{https://stacks.math.columbia.edu/tag/0FCH}{Tag 0FCH}, \href{https://stacks.math.columbia.edu/tag/0FCI}{Tag 0FCI}]{stacks-project}, both  $N$ and $X/X_0$ are Noetherian and Artinian. From \cite[\href{https://stacks.math.columbia.edu/tag/0FCJ}{Tag 0FCJ}]{stacks-project}, they admit composite series. A composite series of $X/X_0$ is equivalent to a filtration $X_0\subset X_1\subset X_2\subset \dots\subset X_n=X$ by subobjects such that $X_i/X_{i-1}$ is simple for every $1\le i\le n$. This filtration and every composite  series of $N$ glue to a composite series of $X$
	with a step $N\subset X_0$, whose factor is isomorphic to $Y$. By the Jordan-Hölder lemma \cite[\href{https://stacks.math.columbia.edu/tag/0FCK}{Tag 0FCK}]{stacks-project},  up to isomorphism $Y$ has finitely many choices.
	\item One may assume  $X\neq0$. Let $\cP$ be the family of nonzero semisimple subobjects of $X$. As $X$ is Artinian, it has a  simple subobject, so $\cP$ is nonempty. Since $X$ is Noetherian,  $\cP$ has a maximal element $i:X_0\to X$. By assumption, there is a subobject $F\subset X$ with $X_0\oplus F=X$. Then $F=0$. (Otherwise, by \cite[\href{https://stacks.math.columbia.edu/tag/0FCJ}{Tag 0FCJ}]{stacks-project}, $F$ has a nonzero simple subobject $F_0$. Then $X_0\oplus F_0\in \cP$ is strictly larger than $X_0$, which is  a contradiction.) Therefore, $i:X_0\to X$ is an isomorphism, and $X$ is semisimple. \end{enumerate}
\end{proof}
\begin{rk}
In a Noetherian and Artinian abelian category, an object may have infinitely many distinct (non semisimple) subobjects up to isomorphism. This should be compared with Lemma \ref{lm:finisubq} \ref{it:subq=factor}. 
\end{rk}

\begin{lm}\label{lm:tensorLchi}Let $X$ be an algebraic variety over $k$.
	Let $L$ be a $\bar{\Q}_{\ell}$-lisse sheaf of rank $1$ on $X$. Then $-\otimes^LL:D_c^b(X)\to D_c^b(X)$ is an equivalence of categories.  It is t-exact for the perverse t-structures.
\end{lm}
\begin{proof}
	Let $L^{-1}$ be the  lisse sheaf dual to $L$. By associativity of the derived tensor product $\otimes^L$, the pair of functors $(-\otimes^LL,-\otimes^LL^{-1})$ is an equivalence.
\begin{enumerate}\item\label{it:twistLright} Right t-exactness:   The functor is t-exact for the standard t-structures. Thus, for every $K\in {}^pD^{\le0}(X)$ and every integer $n$, one has $\cH^n(K\otimes^LL)=\cH^n(K)\otimes^LL$. Therefore, one has $\Supp \cH^n(K\otimes^LL)=\Supp \cH^n(K)$. Thus, $K\otimes^LL\in  {}^pD^{\le0}(X)$. 

\item Left t-exactness:  By Part \ref{it:twistLright}, for every $K\in {}^pD^{\ge0}(K)$, one has  $L^{-1}\otimes^L\D_XK\in{}^pD^{\le0}(X)$. By \cite[II, Cor.~7.5~f)]{kiehl2001weil},  one has isomorphisms \[\D_X(K\otimes^LL)\to \RHom(L,\D_XK)\to L^{-1}\otimes^L\D_XK\] in $D_c^b(X)$. Therefore, one gets $K\otimes^LL\in  {}^pD^{\ge0}(X)$.\end{enumerate}
\end{proof}
\subsection{Universal local acyclicity}\label{sec:ULA}
 In Section \ref{sec:ULA}, all schemes are assumed to be quasi-compact and quasi-separated. For a scheme $X$ and a geometric point $\bar{x}$ on $X$,  denote by $O_{X,\bar{x}}^{\sh}$ the strict henselization (in the sense of \cite[\href{https://stacks.math.columbia.edu/tag/04GQ}{Tag 04GQ (3)}]{stacks-project}) of $O_{X,\bar{x}}$. Set $X_{(\bar{x})}:=\Spec O_{X,\bar{x}}^{\sh}$.

Let $f:X\to S$ be a separated morphism of finite presentation of schemes over $\Z[1/\ell]$.
\begin{df}[{\cite[\href{https://stacks.math.columbia.edu/tag/0GJM}{Tag 0GJM}]{stacks-project}, \cite[Def.~1.2]{barrett2023singular}}]\label{df:ULA}\index{universally locally acyclic}
	Let $K$ be an object of $D_c^b(X)$.
	\begin{itemize}
		\item If for every geometric point $\bar{x}$ on $X$ and every geometric point $\bar{t}$ on $S_{(\bar{s})}$ with $\bar{s}=f(\bar{x})$, the canonical morphism \[R\Gamma(X_{(\bar{x})},K)\to R\Gamma(X_{(\bar{x})}\times_{S_{(\bar{s})}}\bar t,K)\] is an isomorphism, then $K$ is called \emph{$f$-locally acyclic}.
		\item If for every morphism $S'\to S$ of schemes, in notation of the cartesian square  \begin{equation}\label{tik:cartesian}
			\begin{tikzcd}
				X' \arrow[r, "g'"] \arrow[d, "f'"'] \arrow[rd, "\square", phantom] & X \arrow[d, "f"] \\
				S' \arrow[r, "g"']                                                 & S               
			\end{tikzcd}
		\end{equation} $g'^*K$ is $f'$-locally acyclic, then $K$ is called \emph{$f$-universally locally acyclic} ($f$-ULA). Let $D^{\ULA}(X/S)\subset D_c^b(X)$ be the full subcategory of $f$-ULA objects.\end{itemize}
\end{df}
By \cite[Thm.~4.4]{hansen2023relative}, an object $K\in D_c^b(X)$ is $f$-ULA if and only if $K$ is  universally locally acyclic in the sense of \cite[Def.~3.2]{hansen2023relative}. Thus, the notation $D^{\ULA}(X/S)$ agrees with that in \cite{hansen2023relative}. It is a triangulated subcategory of $D_c^b(X)$. 
\begin{ft}\label{ft:ULA}
\hfill\begin{enumerate}\item\label{it:ULAoverfld} \textup{(\cite[Lem.~3.7~(ii)]{barrett2023singular})} If $S=\Spec k$, then $D^{\ULA}(X/k)=D_c^b(X)$.
	\item\label{it:ULAoverid} \textup{(\cite[Lem.~3.7~(i)]{barrett2023singular})} If $f:X\to S$ is an isomorphism, then $D^{\ULA}(X/S)\subset D_c^b(X)$ is the full subcategory  of objects whose cohomology sheaves are lisse.
\item\label{it:basechgULA} \textup{(\cite[Prop.~3.4 (i)]{hansen2023relative})} Let $g:S'\to S$ be a morphism of schemes over $\Z[1/\ell]$. Then in the notation of \eqref{tik:cartesian}, the functor $g'^*:D_c^b(X)\to D_c^b(X')$ restricts to a functor $D^{\ULA}(X/S) \to D^{\ULA}(X'/S')$.
\item \label{it:properULA} \textup{(\cite[Lem.~3.15]{richarz2014new}, \cite[Lem.~3.6 (i), (ii)]{barrett2023singular})} Let $f':Y\to S$ be a separated morphism of finite presentation of schemes over $\Z[1/\ell]$. Let $h:X\to Y$ be a morphism of schemes over $S$. If $h$ is smooth (resp. proper), then the functor $h^*:D_c^b(Y)\to D_c^b(X)$ (resp. $Rh_*:D_c^b(X)\to D_c^b(Y)$) restricts to a functor $D^{\ULA}(Y/S)\to D^{\ULA}(X/S)$  (resp. $D^{\ULA}(X/S)\to D^{\ULA}(Y/S)$).
\item\label{it:composeULA} \textup{(\cite[p.643]{hansen2023relative})} Let $g:S\to T$ be a smooth morphism of schemes over $\Z[1/\ell]$. Then $D^{\ULA}(X/S)\subset D^{\ULA}(X/T)$. 
\item\label{it:externaltensorULA} \textup{(\cite[Thm.~A.2.5 (4)]{zhu2016introduction})} For $i=1,2$, let $f_i:X_i\to S$  be a separated morphism of finite presentation of schemes over $\Z[1/\ell]$, and let $K_i\in D^{\ULA}(X_i/S)$. Then $K_1\boxtimes_SK_2\in D^{\ULA}(X_1\times_SX_2/S)$.\end{enumerate}
\end{ft}

\begin{lm}\label{lm:ULAdeterminedbygene}
	Assume that $S$ is Noetherian, irreducible with generic point $\eta$. Let $K\in D^{\ULA}(X/S)$. If $K|_{X_{\bar{\eta}}}=0$ in $D_c^b(X_{\bar{\eta}})$, then $K=0$.
\end{lm}
\begin{proof}	It suffices to prove that for every geometric point $\bar s$ on $S$, one has $K|_{X_{\bar{s}}}=0$ in $D_c^b(X_{\bar{s}})$. By \cite[Prop.~7.1.9]{EGAII}, as $S$ is Noetherian, there is a discrete valuation ring $R$ and a separated morphism $g:\Spec(R)=S'\to S$, sending the generic (resp. closed) point $\xi$ (resp. $r$) of $S'$ to $\eta$ (resp. $s$). Let $i:R\to R^h$ be the henselization of $R$ (in the sense of \cite[\href{https://stacks.math.columbia.edu/tag/04GQ}{Tag 04GQ (1)}]{stacks-project}). By \cite[\href{https://stacks.math.columbia.edu/tag/0AP3}{Tag 0AP3}]{stacks-project}, $R^h$ is a discrete valuation ring. From \cite[I, Exercise 4.9]{milne2016etale}, the local morphism $i:R\to R^h$ is injective. Then $i^*:\Spec(R^h)\to S'$ preserves the generic (resp. closed) point. Replacing $R$ by $R^h$, one may assume further that $R$ is henselian. 
	
	 Consider the following cartesian squares \begin{center}
		\begin{tikzcd}
			X'_{\bar{r}} \arrow[r, "\bar{i}"] \arrow[d] \arrow[rd, "\square", phantom] & X'_{(\bar{r})} \arrow[d] \arrow[rd, "\square", phantom] & X'_{\bar{\xi}} \arrow[l, "\bar{j}"'] \arrow[d] \\
			\bar{r} \arrow[r]                                                         & S'_{(\bar{r})}                                          & \bar{\xi}, \arrow[l]                          
		\end{tikzcd}
	\end{center}where every vertical morphism is a base change of $f:X\to S$. In the notation of \eqref{tik:cartesian},  let $R\Phi:D^+(X')\to D^+(X'_{\bar{r}})$ be the vanishing cycle functor. Let $R\Psi:D^+(X')\to D^+(X'_{\bar{r}})$ be the nearby cycle functor. 
Set $K'=g'^*K\in D_c^b(X')$.  By definition, one has $R\Psi(K')=\bar{i}^*R\bar{j}_*(K'|_{X'_{\bar{\xi}}})$.  From \cite[(1.1.3)]{illusie2006vanishing}, as $R$  is henselian, there is a natural exact triangle $K'|_{X'_{\bar{r}}} \to R\Psi (K') \to R\Phi (K')\overset{+1}{\to}$ in $D^+(X'_{\bar{r}})$. Since $K'|_{X'_{\bar{\xi}}}$ is a pullback of $K|_{X_{\bar{\eta}}}=0$, one has $K'|_{X'_{\bar{\xi}}}=0$ and hence $R\Psi (K')=0$. By \cite[Cor.~3.5]{illusie2006vanishing}, the universal local acyclicity of $K$ implies $R\Phi (K')=0$. Therefore, one gets $K'|_{X'_{\bar{r}}} =0$. 

By Lemma \ref{lm:pervext} \ref{it:JKLMA1}, since $K'|_{X'_{\bar{r}}} $ is the pullback of $K|_{X_{\bar{s}}}$ under the field extension $k(\bar{r})/k(\bar{s})$,  one gets $K|_{X_{\bar{s}}}=0$.
			\end{proof}
\subsection{Relative perverse sheaves}
Let $f:X\to S$ be a morphism of algebraic varieties over a field $k$.  In particular, $f$ is separated and of finite presentation. As $S$ is \emph{bon} in the sense of \cite[(1.0)]{katz1985transformation}, one can consider $K_{X/S}:=Rf^!\bar{\Q}_{\ell}\in D_c^b(X)$ the relative dualizing complex. The  functor \[\D_{X/S}(-)=\RHom(-,K_{X/S}):D_c^b(X)\to D_c^b(X)^{\op}\] is called the \emph{relative Verdier dual}\index{relative Verdier dual functor}. By \cite[(1.1.5)]{katz1985transformation}, as $S$ is bon, there is a canonical morphism of functors $\Id_{D_c^b(X)}\to \D_{X/S}\circ \D_{X/S}$. 
Fact \ref{ft:relperv} is stated for $\infty$-categories in \cite{hansen2023relative}, but holds for the underlying triangulated categories (described in \cite[Lem.~7.9]{hemo2023constructible}) by \cite[Footnote 1]{hansen2023relative}. \begin{ft}\label{ft:relperv}\index{relative perverse sheaves}
\hfill\begin{enumerate}	
	\item\label{it:onDcb} \textup{(\cite[Thm.~1.1]{hansen2023relative})} There is a unique t-structure $({}^{p/S}D^{\le0}(X/S),{}^{p/S}D^{\ge0}(X/S))$ on $D_c^b(X)$, called the relative perverse t-structure, with the following property: An object $K\in D_c^b(X)$ lies in ${}^{p/S}D^{\le0}(X/S)$ (resp. ${}^{p/S}D^{\ge0}(X/S)$) if and only if for every geometric point $\bar{s}\to S$, the restriction $K|_{X_{\bar{s}}}$ lies in ${}^pD^{\le0}(X_{\bar{s}})$ (resp. ${}^pD^{\ge0}(X_{\bar{s}})$). In particular, for every $s\in S$, the  functor $(-)|_{X_s}:D_c^b(X)\to D_c^b(X_s)$ is  t-exact, where the source (resp. target) is equipped with the relative (resp. absolute) perverse t-structure. 
 	\item\label{it:resrelt} \textup{(\cite[Thm.~1.9]{hansen2023relative})} The relative perverse t-structure on $D_c^b(X)$ restricts to a t-structure $( {}^{p/S}D^{\ULA,\le0}(X/S) ,{}^{p/S}D^{\ULA,\ge0}(X/S) )$ on $D^{\ULA}(X/S)$.
	\item\label{it:relVerdier} \textup{(\cite[Prop.~3.4]{hansen2023relative})} The functor $\D_{X/S}$ preserves $D^{\ULA}(X/S)$, and the morphism $\Id_{D^{\ULA}(X/S)}\to \D_{X/S}\circ \D_{X/S}$ of functors $D^{\ULA}(X/S)\to D^{\ULA}(X/S)$  is an isomorphism. The formation of $\D_{X/S}:D^{\ULA}(X/S)\to D^{\ULA}(X/S)^{\op}$ commutes with any base change in $S$, so $\D_{X/S}$ exchanges ${}^{p/S}D^{\ULA,\le0}(X/S)$ with ${}^{p/S}D^{\ULA,\ge0}(X/S)$.
\end{enumerate}
	\end{ft}	
\begin{df}\label{df:relperv}Let $\Perv(X/S)$ (resp. $\Perv^{\ULA}(X/S)$) be the heart of the relative perverse t-structure on $D_c^b(X)$ (resp. $D^{\ULA}(X/S)$).\end{df} By Fact \ref{ft:relperv} \ref{it:onDcb}, an object $K\in D_c^b(X)$ lies in $\Perv(X/S)$ if and only if for every geometric point $\bar{s}\to S$, one has $K|_{X_{\bar{s}}}\in \Perv(X_{\bar{s}})$.
\begin{eg}\label{eg:relperv}\hfill\begin{enumerate}\item\label{it:extreme} \textup{(\cite[p.632]{hansen2023relative})} If $S=\Spec(k)$, then $\Perv(X/k)=\Perv(X)$.
		\item  If $f:X\to S$ is universally injective, then $\Perv(X/S)=\Cons(X)$.
\item \textup{(\cite[Lem.~3.7 (ii)]{barrett2023singular})} If $f:X\to S$ is smooth of relative dimension $r$, then the functor $(-)[r]:\Loc(X)\to D_c^b(X)$ factors through $\Perv^{\ULA}(X/S)$.\end{enumerate}\end{eg}

\begin{eg}\label{eg:subisULA}
	Let $i:Y\to X$ be a closed immersion of schemes over $S$. Assume that $Y\to S$ is smooth of relative dimension $d$ and with geometrically connected fibers. If $L$ is a $\bar{\Q}_{\ell}$-lisse sheaf on $Y$, then $i_*L[d]\in\Perv^{\ULA}(X/S)$.
	
	Indeed, by Fact \ref{ft:ULA} \ref{it:ULAoverid}, one has $L\in D^{\ULA}(Y/Y)$. From the smoothness of $Y\to S$ and Fact \ref{ft:ULA} \ref{it:composeULA}, one has $L\in D^{\ULA}(Y/S)$. Using the properness of $i:Y\to X$ and Fact \ref{ft:ULA} \ref{it:properULA}, one has $i_*L[d]\in D^{\ULA}(X/S)$. For every geometric point $\bar{s}\to S$, let $i_{\bar{s}}:Y_{\bar{s}}\to X_{\bar{s}}$ be the base change of $i$ along the morphism $X_{\bar{s}}\to X$. By the proper base change theorem, one has $i_*L[d]|_{X_{\bar{s}}}=(i_{\bar{s}})_*(L|_{Y_{\bar{s}}})[d]$ which is in $\Perv(X_{\bar{s}})$. Therefore, $i_*L[d]\in \Perv^{\ULA}(X/S)$.
\end{eg}\begin{lm}
If $S$ is geometrically unibranch and irreducible, then  $\Perv^{\ULA}(X/S)$ is a Serre subcategory of $\Perv(X/S)$. 
\end{lm}
\begin{proof}By definition, $\Perv^{\ULA}(X/S)$ is a strictly full subcategory of $\Perv(X/S)$. 
By  Fact \ref{ft:relperv} \ref{it:resrelt} and \cite[Thm.~1.3.6]{beilinson2018faisceaux}, $\Perv^{\ULA}(X/S)$ is an  abelian subcategory of $\Perv(X/S)$ and closed under extensions in $D^{\ULA}(X/S)$. As $D^{\ULA}(X/S)\subset D_c^b(X)$ is a triangulated subcategory, $\Perv^{\ULA}(X/S)$ is closed under extensions in $\Perv(X/S)$. From the proof of \cite[Thm.~6.8 (ii)]{hansen2023relative}, because $S$ is  geometrically unibranch,  $\Perv^{\ULA}(X/S)$ is closed under subquotients in $\Perv(X/S)$. By \cite[\href{https://stacks.math.columbia.edu/tag/02MP}{Tag 02MP}]{stacks-project}, it is a Serre subcategory.
\end{proof}
\begin{ft}[{\cite[Thm.~1.10~(ii)]{hansen2023relative}}]\label{ft:relULAingeneric}Assume that $S$ is geometrically unibranch and irreducible with generic point $\eta$.  Then the functor \[(-)|_{X_{\eta}}:\Perv^{\ULA}(X/S)\to \Perv(X_{\eta})\] is exact and fully faithful, and its essential image is stable under subquotients.
\end{ft}
For every integer $j$, let ${}^{p/S}\cH^j:D_c^b(X)\to \Perv(X/S)$ be the $j$-th cohomology functor associated with the relative perverse t-structure.
\begin{lm}\label{lm:genericfibersemisimple}Suppose that $S$ is smooth over $k$ and irreducible with generic point $\eta$.
	Assume that $K\in \Perv^{\ULA}(X/S)$ is semisimple in $D_c^b(X)$. Then $K|_{X_{\eta}}$ is semisimple in $\Perv(X_{\eta})$.\end{lm}
\begin{proof}By   Fact \ref{ft:relULAingeneric},  for every subobject $M\subset K|_{X_{\eta}}$ in $\Perv(X_{\eta})$,  there is a subobject $K'\subset K$ in $\Perv^{\ULA}(X/S)$ with $K'|_{X_{\eta}}=M$. By Lemma \ref{lm:shift} and smoothness of $S$, the morphism $K'[\dim S]\to K[\dim S]$ is a monomorphism in $\Perv(X)$. Because $K$ is semisimple in $D_c^b(X)$, its shift $K[\dim S]$ is semisimple in $\Perv(X)$. Thus, there is  a subobject $N\subset K[\dim S]$ in $\Perv(X)$ with $K[\dim S]=(K'[\dim S])\oplus N$. Then $K=K'\oplus (N[-\dim S])$ in $D_c^b(X)$.  For every nonzero integer $j$, one has \[0={}^{p/S}\cH^j(K)=0\oplus {}^{p/S}\cH^j(N[-\dim S])\] in $\Perv(X/S)$. Hence ${}^{p/S}\cH^j(N[-\dim S])=0$ and  $N[-\dim S]\in\Perv(X/S)$. Consequently, $K|_{X_{\eta}}=M\oplus (N|_{X_{\eta}}[-\dim S])$ in $\Perv(X_{\eta})$. By \cite[Thm.~4.3.1 (i)]{beilinson2018faisceaux}, the abelian category $\Perv(X_{\eta})$ is Noetherian and Artinian. As every subobject of $K|_{X_{\eta}}$ in $\Perv(X_{\eta})$  admits a direct complement, the semisimplicity follows from Lemma \ref{lm:finisubq} \ref{it:subhascom}.\end{proof}

Lemma \ref{lm:shift} is stated without proof  for regular schemes $S$ in \cite[p.636]{hansen2023relative}.
\begin{lm}\label{lm:shift}
 Assume that $S$ is smooth over $k$ of equidimension $d$. Then the shifted inclusion \begin{equation}\label{eq:shift}(-)[d]:(D^{\ULA}(X/S),{}^{p/S}D^{\ULA,\le0}(X/S),{}^{p/S}D^{\ULA,\ge0}(X/S))\to (D_c^b(X),{}^pD^{\le0}(X),{}^pD^{\ge0}(X))\end{equation} 
 is t-exact.  In particular, it restricts to an exact functor \begin{equation}\label{eq:shiftperv}
 (-)[d]:\Perv^{\ULA}(X/S)\to \Perv(X),
 \end{equation}and a t-exact equivalence
\[(-)[d]: (D^{\ULA}(X/S),{}^{p/S}D^{\ULA,\le0}(X/S),{}^{p/S}D^{\ULA,\ge0}(X/S))\to (D^{\ULA}(X/S),{}^pD^{\le0}(X),{}^pD^{\ge0}(X)).\]
\end{lm}
\begin{proof}
\begin{enumerate} \item\label{it:dshiftrightext} We prove that the functor \[(-)[d]:(	D_c^b(X),{}^{p/S}D^{\le0}(X/S),{}^{p/S}D^{\ge0}(X/S))\to (D_c^b(X),{}^pD^{\le0}(X),{}^pD^{\ge0}(X))\] is right t-exact. For  every geometric point $\bar{s}$ on $S$,  the functor $(-)|_{X_{\bar{s}} }:D_c^b(X)\to D_c^b(X_{\bar{s}} )$ is t-exact for the standard t-structures. Then 
for every integer $n$ and every $K\in {}^{p/S}D^{\le0}(X/S)$,  one has $\cH^n(K[d])|_{X_{\bar{s}}}=\cH^{n+d}(K|_{X_{\bar{s}}})$. Hence \[X_{\bar{s}}\cap \Supp \cH^n(K[d])=\Supp \cH^{n+d}(K|_{X_{\bar{s}}}).\] As $K|_{X_{\bar{s}}}\in {}^pD^{\le0}(X_{\bar{s}})$, one has $\dim \Supp \cH^{n+d}(K|_{X_{\bar{s}}})\le -n-d$. By Lemma \ref{lm:constructibledim} \ref{it:fibdim}, one has \[\dim  \Supp \cH^n(K[d]) \le-n.\] From Lemma \ref{lm:constructibledim} \ref{it:clsdim}, the Zariski closure of $\Supp \cH^n(K[d])$ in $X$ has dimension at most $-n$. Hence $K[d]\in {}^pD^{\le0}(X)$. 

\item The functor (\ref{eq:shift}) is left t-exact.  By the proof of \cite[Lemma ~3.12]{barrett2023singular}, as $S$ is smooth,
for every $K\in{}^{p/S}D^{\ULA,\ge0}(X/S) $,  one has $\D_X(K[d])=(\D_{X/S}K)[d](d)$  in  $D_c^b(X)$.  From Fact \ref{ft:relperv} \ref{it:relVerdier}, one has $\D_{X/S}K\in {}^{p/S}D^{\ULA,\le0}(X/S)$. By Part \ref{it:dshiftrightext}, one has $(\D_{X/S}K)[d]\in {}^pD^{\le0}(X)$ and hence $K[d]\in {}^pD^{\ge0}(X)$.\end{enumerate}
\end{proof}

By convention, the dimension of an empty space is $-\infty$. 
\begin{lm}\label{lm:constructibledim}	Let $X$ be a scheme of finite type over a field $F$. Let $C$ be a quasi-constructible subset of $X$. 
	\begin{enumerate}
		\item\label{it:clsdim}   Then $\dim C=\dim\bar{C}$.
		\item\label{it:loccls}  Let $\{B_i\}_{i=1}^n$ be finitely many locally closed subsets of $X$. Set $B=\cup_{i=1}^nB_i$. Then $\dim B=\max_{i=1}^n\dim B_i$.
	\end{enumerate}
	Let $f:X\to Y$ be a morphism between schemes of finite type over   $F$. 
	\begin{enumerate}
		\setItemnumber{3}\item\label{it:fibdim} Let $n\ge0$ be an integer such that $\dim (C\cap f^{-1}(y))\le n$ for every $y\in Y$. Then $\dim C\le \dim Y+n$.
		\setItemnumber{4}\item\label{it:generic} Assume that $Y$ is irreducible with generic point $\eta$. Then $\dim Y+\dim (C\cap X_{\eta})\le \dim C$.\end{enumerate}
\end{lm}
\begin{proof}
\hfill \begin{enumerate}\item As $X$ is a Noetherian scheme, the topological subspace $C$ is Noetherian. Therefore, $C$ is the union of finitely many irreducible components. Thus, one may assume further that $C$ is nonempty and irreducible. Then  the reduced induced closed subscheme $\bar{C}$ of $X$ is integral and of finite type over $F$. By \cite[AG. Prop.~1.3]{borel2012linear}, $C$ contains a nonempty open subset of $\bar{C}$. By \cite[II, Exercise 3.20 (e)]{hartshorne2013algebraic}, one has $\dim C=\dim \bar{C}$.
	\item For every $1\le i\le n$, since $B_i\subset B$, one has $\dim B_i\le \dim B$. Then $\max_i\dim B_i\le \dim B$. By Part \ref{it:clsdim}, as every $B_i$ is quasi-constructible in $X$,  one has $\dim B_i=\dim \overline{B_i}$. As $\{\overline{B_i}\}_{i=1}^n$ is a finite closed cover of $\bar{B}$, one gets $\dim B\le \dim \bar{B}=\max_i\dim \overline{B_i}=\max_i\dim B_i$. 
	\item By Part \ref{it:loccls}, one may assume that $C$ is locally closed in $X$. Taking irreducible components, one may assume further that $C$ is irreducible. Let $Z$ be the Zariski closure of $f(C)$ in $Y$. Then $Z$ is irreducible. With reduced induced subscheme structures, one views $C$ and $Z$ as integral schemes of finite type over $F$. Moreover, $f:X\to Y$ induces a dominant morphism $g:C\to Z$ over $F$. Then for every $y\in f(C)=g(C)$, one has \[n\ge\dim C\cap f^{-1}(y)=\dim g^{-1}(y)\labelrel\ge{eq:use3.22} \dim C-\dim Z,\]where \eqref{eq:use3.22} uses \cite[II, Exercise 3.22 (b)]{hartshorne2013algebraic}. Hence $\dim C\le \dim Z+n\le \dim Y+n$.
	\item The statement is topological, so one may assume that $Y$ is reduced. As in the proof of Part \ref{it:fibdim}, one may assume that $C$ is an irreducible, locally closed subset of $X$ and view $C$ as an integral scheme of finite type over $F$. One may assume that $C\cap X_{\eta}$ is nonempty. As $C_{\eta}$ is homeomorphic to $C\cap X_{\eta}$, the morphism $C\to Y$ induced by $f$ is dominant.  By \cite[II, Exercise 3.22 (c)]{hartshorne2013algebraic}, as $Y$ is integral, one has $\dim C\cap X_{\eta}=\dim C_{\eta}=\dim C-\dim Y$. \end{enumerate}
\end{proof}
\begin{lm}\label{lm:genericfiber}
	Assume that $S$ is irreducible with generic point $\eta$ and geometric reduced. Let $d:=\dim S$. Then the functor $(-)|_{X_{\eta}}[-d]:D_c^b(X)\to D_c^b(X_{\eta})$
	is t-exact for the absolute perverse t-structures. In particular, it restricts to an exact functor \begin{equation}\label{eq:genericperv}(-)|_{X_{\eta}}[-d]:\Perv(X)\to \Perv(X_{\eta}).\end{equation}
\end{lm}
\begin{proof}\begin{enumerate}
		\item\label{it:geneshifright} Right t-exactness: For every $K\in {}^pD^{\le0}(X)$ and every integer $n$, one has $\Supp \cH^n(K|_{X_{\eta}}[-d])=\Supp \cH^{n-d}(K|_{X_{\eta}})=X_{\eta}\cap \Supp \cH^{n-d}(K)$. By Lemma \ref{lm:constructibledim} \ref{it:generic}, one has \[\dim \Supp \cH^n(K|_{X_{\eta}}[-d])\le \dim  \Supp(\cH^{n-d}(K))-d\le-n.\] From Lemma \ref{lm:constructibledim} \ref{it:clsdim}, one has $K|_{X_{\eta}}[-d]\in{}^pD^{\le0}(X_{\eta})$.
		\item To prove left t-exactness, one may shrink $S$ to a nonempty open. By \cite[\href{https://stacks.math.columbia.edu/tag/056V}{Tag 056V}]{stacks-project}, as $S$ is geometrically reduced, shrinking $S$ one may assume that $S$ is smooth. From \cite[Thm.~2.13, p.242]{SGA4.5},  shrinking $S$  one may assume $K\in D^{\ULA}(X/S)$. Then for every integer $n$, we prove  \begin{equation}\label{eq:supptwist}\Supp \cH^n\left(\D_{X_{\eta}}(K|_{X_{\eta}}[-d])\right)=\Supp \cH^n\left((\D_XK)|_{X_{\eta}}[-d]\right).\end{equation}Indeed,  By the proof of \cite[Lem.~3.12]{barrett2023singular}, as $S$ is smooth, one has $\D_XK=(\D_{X/S}K)(d)[2d]$. From Fact \ref{ft:relperv} \ref{it:relVerdier}, as $K$ is ULA, $(\D_XK)|_{X_{\eta}}[-d]$ is a Tate twist of $\D_{X_{\eta}}(K|_{X_{\eta}}[-d])$, which proves \eqref{eq:supptwist}. 
		
		Now assume $K\in{}^pD^{\ge0}(X)$. Then $\D_XK\in {}^pD^{\le0}(X)$. From Part \ref{it:geneshifright}, one has $(\D_XK)|_{X_{\eta}}[-d]\in {}^pD^{\le0}(X_{\eta})$. By (\ref{eq:supptwist}), one has $\D_{X_{\eta}}(K|_{X_{\eta}}[-d])\in {}^pD^{\le0}(X_{\eta})$, or equivalently, $K|_{X_{\eta}}[-d]\in {}^pD^{\ge0}(X_{\eta})$. 
	\end{enumerate}
\end{proof}
\begin{lm}\label{lm:Scholze}
Assume that $S$ is smooth over $k$, integral with generic point $\eta$ and  $\dim S=d$.  Then:
\begin{enumerate}
\item\label{it:Scholzekey} Let $A\in \Perv^{\ULA}(X/S)$, and  let  $B[d]$ be a subquotient of $A[d]$ in $\Perv(X)$. If the image $B|_{X_{\eta}}\in\Perv(X_{\eta})$ of $B[d]$ under the functor (\ref{eq:genericperv}) is zero, then $B[d]=0$ in $\Perv(X)$.
\item\label{it:ScholzeSerre} The functor \eqref{eq:shiftperv} identifies $\Perv^{\ULA}(X/S)$ as a Serre subcategory of $\Perv(X)$.
\end{enumerate} 
\end{lm}\begin{proof}
\hfill\begin{enumerate}
\item By regularity of $S$ and \cite[Cor.~1.12]{hansen2023relative}, one has $B\in D^{\ULA}(X/S)$. Then by Lemma \ref{lm:ULAdeterminedbygene}, since $B|_{X_{\eta}}=0$,  one has $B=0$.
\item It follows from the definition that the functor  \eqref{eq:shiftperv} is  fully faithful. Its  essential image is closed under extensions in $\Perv(X)$, because  $\Perv^{\ULA}(X/S)$ is closed under extensions in the triangulated subcategory $D^{\ULA}(X/S)$ of $D_c^b(X)$. 

We claim that the essential image is closed under taking subobjects. Take $K\in \Perv^{\ULA}(X/S)$ and a subobject $L[d]$ of $K[d]$ in $\Perv(X)$. By Lemma \ref{lm:genericfiber}, as $S$ is integral and smooth,    $L|_{X_{\eta}}$  is a subobject of $K|_{X_{\eta}}$ in $\Perv(X_{\eta})$. By smoothness of $S$ and Fact \ref{ft:relULAingeneric}, there is a subobject $L'\subset K$ in $\Perv^{\ULA}(X/S)$ with $L'|_{X_{\eta}}=L|_{X_{\eta}}$. Set $M=K/L'\in\Perv^{\ULA}(X/S)$. Let $N[d]$ be the image of $L[d]$ under the morphism $K[d]\to M[d]$ in $\Perv(X)$. From Lemma \ref{lm:genericfiber}, as the sequence \[0\to L'[d]\cap L[d]\to L[d]\to N[d]\to0\] is exact in $\Perv(X)$,  the sequence \[0\to L'|_{X_{\eta}}\cap L|_{X_{\eta}}\to L|_{X_{\eta}}\to N|_{X_{\eta}}\to0\] is exact in $\Perv(X_{\eta})$. It implies $N|_{X_{\eta}}=0$. By Part \ref{it:Scholzekey},  since $N[d]$ is a subobject of $M[d]$ in $\Perv(X)$, one has $N[d]=0$. Then $L[d]$ is a subobject of $L'[d]$ in $\Perv(X)$. Since $(L'[d])/(L[d])$ is a quotient of $L'[d]$ in $\Perv(X)$ and $(L'|_{X_{\eta}})/(L|_{X_{\eta}})=0$ in $\Perv(X_{\eta})$, one gets $(L'[d])/(L[d])=0$ in $\Perv(X)$. Therefore, $L[d]=L'[d]$.  The claim is proved. 

Similarly, the essential image is closed under taking quotients.  By \cite[\href{https://stacks.math.columbia.edu/tag/02MP}{Tag 02MP}]{stacks-project}, the essential image is a Serre subcategory of $\Perv(X)$.
\end{enumerate}
\end{proof}

	\section{Cotori}\label{sec:cotori}We review Gabber and Loeser's construction of a scheme structure on the set of pro-$\ell$ characters.
 \subsection{Definition and basic properties}By \cite[p.127]{robert2000course}, there is a canonical absolute value on $\bar{\Q}_{\ell}$ extending the discrete absolute value $|\cdot|_{\ell}$ on $\Q_{\ell}$. It induces a topology on $\bar{\Q}_{\ell}$ which is totally disconnected. A subset $A\subset\bar{\Q}_{\ell}$ is closed if and only if for every finite subextension $E/\Q_{\ell}$ of $\bar{\Q}_{\ell}$, the subset $A\cap E$ is closed in the discrete valuation field $E$.
 
 For a profinite group $G$, let $\cC(G)$ be the group of $\ell$-adic characters, \textit{i.e.}, continuous morphisms $\chi:G\to \bar{\Q}_{\ell}^{\times}$. Then $\chi(G)$ are compact subgroup of $\bar{\Q}_{\ell}^{\times}$.
 \begin{lm}\label{lm:cptinL}
 	\hfill	\begin{enumerate}\item\label{it:cptinE}	Let $C$ be a  compact subset of $\bar{\Q}_{\ell}$. Then there is a finite subextension $E$ of $\bar{\Q}_{\ell}/\Q_{\ell}$ with $C\subset E$. 
 		\item\label{it:cptinOE} Let $G\le \bar{\Q}_{\ell}^{\times}$ be a compact subgroup. Then  there is a finite subextension $E$ of $\bar{\Q}_{\ell}/\Q_{\ell}$ with $G\subset O_E^{\times}$.
 		\item\label{it:cptdec} In Part \ref{it:cptinOE}, let $G^{(\ell)}$ (resp. $G^{(\ell')}$) be the $\ell$-Sylow subgroup (resp. maximal prime-to-$\ell$ quotient) of $G$. Then as topological group $G$ is isomorphic to $G^{(\ell)}\times G^{(\ell')}$, and $G^{(\ell')}$  is  finite.\end{enumerate}
 \end{lm}
 \begin{proof}
 	\begin{enumerate}
 		\item Assume the contrary. Then there is a sequence of elements $x_1,x_2,\dots$ in $C$ with $[\Q_{\ell}(x_{n+1}):\Q_{\ell}]>[\Q_{\ell}(x_n):\Q_{\ell}]$ for every integer $n>0$. Let $B\subset C$ be the (infinite) set of elements of this sequence. For every subset $S\subset B$, every  finite subextension $F/\Q_{\ell}$, the set $S\cap F$ is finite, so closed in $F$. Therefore, $S$ is closed in $\bar{\Q}_{\ell}$. In particular, the  set $B$ is closed and hence compact in $C$.  Every subset of $B$ is closed in $B$, so $B$ is  discrete. Thus, $B$ is finite, a contradiction.
 		\item By Part \ref{it:cptinE}, there is a finite subextension $E$ of $\bar{\Q}_{\ell}/\Q_{\ell}$ containing $G$. By \cite[Thm.~1 2, p.122]{serre2009lie}, one has $G\subset O_E^{\times}$. 
 		\item By Part \ref{it:cptinOE} and \cite[Cor., p.155]{serre2009lie}, $G$  is an $\ell$-adic Lie group. From Lazard's theorem (see, \textit{e.g.}, \cite[p.711]{gonzalez2009analytic}), there is a pro-$\ell$ open subgroup $U$ of $G$. By \cite[Cor.~2.3.6 (b)]{ribes2000profinite},  there is an $\ell$-Sylow subgroup $H\le G$ containing $U$. Since $G$ is compact, $[G:U]$ is finite. Thus, the group $G/H$ is finite of order prime to $\ell$. By \cite[Prop.~2.3.8]{ribes2000profinite}, $G$ is isomorphic to $G/H\times H$. By \cite[Cor.~2.3.6 (c)]{ribes2000profinite}, since $G$ is commutative, it has exactly one $\ell$-Sylow subgroup. 
 	\end{enumerate}
 \end{proof} 	Let $\cC(G)_{\ell'}$ (resp. $\cC(G)_{\ell}$) be the subgroup of $\cC(G)$ consisting of characters of finite order prime to $\ell$ (resp. that are pro-$\ell$).  Then there is a canonical isomorphism $\cC(G)_{\ell}\iso \cC\big((G^{(\ell)})^{\ab}\big)$. By Lemma \ref{lm:cptinL} \ref{it:cptdec},  one has  $
 \cC(G)=\cC(G)_{\ell'}\times \cC(G)_{\ell}$.
 
We review the contents of \cite[Sec.~3.2]{gabber1996faisceaux}. Fix	 an integer $n\ge 0$.  Let $A_n$ be a free $\hat{\Z}$-module of rank $n$. Let $\{\gamma_1,\dots,\gamma_n\}$ be a $\Z_{\ell}$-basis of $A_n^{(\ell)}$. Let $\cR=\{O_E:E/\Q_{\ell}\text{ is a finite subextension of }\bar{\Q}_{\ell}\}$, which is a directed set under inclusion. For every $R\in \cR$, let $m_R$ be the maximal ideal of $R$. Let $R[[A_n^{(\ell)}]]:=\varprojlim_{i,j\ge1}(R/m_R^i)[A_n^{(\ell)}/\ell^j]$ be the completed group ring.   There is a canonical injective morphism $A_n^{(\ell)}\to R[[A_n^{(\ell)}]]^{\times}$ of groups.

\begin{ft}[{\cite[p.509]{gabber1996faisceaux}}] The ring $R[[A_n^{(\ell)}]]$ is a Noetherian, regular, complete, local domain  of Krull dimension $1+n$. There is	an  isomorphism of topological rings \begin{equation}\label{eq:Iwasawaiso}
		R[[A_n^{(\ell)}]]\to R[[X_1,\dots,X_n]],\quad \gamma_i\mapsto 1+X_i.
\end{equation}\end{ft} 

Gabber and Loeser introduce a scheme of $\ell$-adic characters. 
\begin{df}\label{df:cotorus}Write $R_n=\bar{\Q}_{\ell}\otimes_{\Z_{\ell}}\Z_{\ell}[[A_n^{(\ell)}]]$. Define  the ``cotorus"\index{cotorus} associated with  $A_n$ to be $\cC_{\ell}:=\Spec R_n$.\end{df}  By \cite[Prop.~A.2.2.3 (ii)]{gabber1996faisceaux}, the scheme $\cC_{\ell}$ is integral, Noetherian and regular. Its set of closed points coincides with $\cC_{\ell}(\bar{\Q}_{\ell})$, and it is Zariski dense in $\cC_{\ell}$. If $n>0$, then $\cC_{\ell}$ is \textbf{not} locally of finite type over $\bar{\Q}_{\ell}$. 
\begin{lm}\label{lm:extendcharacter}
	Every character $\chi:A_n^{(\ell)}\to \bar{\Q}_{\ell}^{\times}$ extends canonically to a surjective morphism $R_n\to \bar{\Q}_{\ell}$ of $\bar{\Q}_{\ell}$-algebras.
\end{lm}
\begin{proof}
	There is a finite subextension $E/\Q_{\ell}$ in $\bar{\Q}_{\ell}$ containing all the $\chi(\gamma_i)$. 
	Then by completeness of $E$, for every $f=\sum_{\alpha\in \N^n}c_{\alpha}X^{\alpha}\in \Z_{\ell}[[X_1,\dots,X_n]]$, the series $\sum_{\alpha\in \N^n}c_{\alpha}\prod_{i=1}^n(\chi(\gamma_i)-1)^{\alpha_i}$ converges in $E$. Denote its limit by $f(\chi(\gamma_1)-1,\dots,\chi(\gamma_n)-1)$. The composition $\Z_{\ell}[[A_n^{(\ell)}]]\to E$ of   \eqref{eq:Iwasawaiso} followed by \[\Z_{\ell}[[X_1,\dots,X_n]]\to E, \quad f\mapsto f(\chi(\gamma_1)-1,\dots,\chi(\gamma_n)-1)\] extends $\chi$. It induces the stated surjection. The construction is independent of the choice of the $\Z_{\ell}$-basis $\{\gamma_1,\dots,\gamma_n\}$ of $A_n^{(\ell)}$.	
\end{proof} 
 By Lemma \ref{lm:extendcharacter}, for every $\chi\in \cC(A_n)_{\ell}$, the corresponding character $A_n^{(\ell)}\to \bar{\Q}_{\ell}^{\times}$ induces a surjection $R_n\to \bar{\Q}_{\ell}$. Let $\Psi(\chi)$  be the corresponding element of $\cC_{\ell}(\bar{\Q}_{\ell})$. Thus, there is a map \begin{equation}
	\label{eq:chartoscheme}\Psi: \cC(A_n)_{\ell}\to \cC_{\ell}(\bar{\Q}_{\ell}).
\end{equation} 
\begin{ft}[{\cite[p.519]{gabber1996faisceaux}}]
	The map  (\ref{eq:chartoscheme})  is  bijective. \end{ft}

\subsection{Cotori are Baire}\label{sec:cotori>Baire}
Fix an uncountable, algebraically closed field $k$. The objective  of Section \ref{sec:cotori>Baire} is Lemma \ref{lm:XnBaire}, used in the proof of Theorem \ref{thm:normal}. We  show that over $k$, a reasonable scheme has uncountably many rational points outside a countable union of strict closed subsets.  
\subsubsection*{Baire schemes}
\begin{df}\label{df:k-Baire}A scheme $X$ over $k$ is called \emph{$k$-Baire}, if its dimension $\dim X$ is finite and
	$X(k)\setminus\cup_{i\ge 1}Z_i(k)$ is uncountable for every countable sequence   $\{Z_i\}_{i\ge1}$ of closed subschemes of $X$ with $\dim Z_i<\dim X$ for all $i$. A $k$-algebra $R$ is called $k$-Baire if   $\Spec(R)$ is $k$-Baire. \end{df}\begin{eg}
	Assume $k=\C$. Let $X$ be a complex algebraic variety with $\dim X>0$. The analytification $X^{\an}$ of $X$ is locally compact Hausdorff. Then by  the Baire category theorem (see, \textit{e,g.}, \cite[Cor.~25.4 a)]{willard70general}), $X$ is  $\C$-Baire.
\end{eg}
\begin{rk}\label{rk:Bairebasic}An algebraic curve over $k$ is $k$-Baire. In Definition \ref{df:k-Baire}, the underlying reduced induced closed subscheme $X_{\red}\to X$ induces a bijection $X_{\red}(k)\to  X(k)$, so $X$ is $k$-Baire if and only $X_{\red}$ is $k$-Baire.  If $X$ is irreducible and $k$-Baire, then $X\setminus\cup_{i\ge 1}Z_i$ is Zariski dense in $X$. \end{rk}
\begin{lm}\label{lm:upBaire}
	Let $f:X\to Y$ be a finite surjective morphism of schemes over $k$. If $Y$ is $k$-Baire, then so is $X$.
\end{lm}
\begin{proof}
	Let $\{Z_i\}_i$ be a sequence of closed subschemes of $X$ with $\dim Z_i<\dim X$. Since $f$ is a closed morphism, every $Y_i:=f(Z_i)$ is closed in $Y$.
	Endow each $Y_i$ with the reduced induced structure. Let $Z'_i:=f^{-1}(Y_i)=Y_i\times_YX$. Then there is a canonical closed immersion $Z_i\to Z'_i$.   The restriction  $Z_i\to Y_i$ of $f$ is a finite surjective morphism. By \cite[\href{https://stacks.math.columbia.edu/tag/0ECG}{Tag 0ECG}]{stacks-project}, one has $\dim X=\dim Y$ and $\dim Y_i= \dim Z_i$. In particular, $\dim X$ is finite and $\dim Y_i<\dim Y$.   
	
	As $k$ is algebraically closed, the induced map $X(k)\to Y(k)$ is surjective.  Then the induced map \[X(k)\setminus(\cup_{i\ge 1}Z'_i(k))\to Y(k)\setminus(\cup_iY_i(k))\] is surjective. Because $Y$ is $k$-Baire, the target is uncountable. 
	Then $X(k)\setminus(\cup_{i\ge 1}Z_i(k))$ is also uncountable, as it contains the source.
\end{proof}
\begin{lm}\label{lm:Bairecomp}
	Let $X$ be a Noetherian  scheme over $k$.
	\begin{enumerate}\item\label{it:component} Then $X$ is $k$-Baire if and only if $X$ has an  irreducible component $C$ with $\dim C=\dim X$, such that the underlying reduced induced closed subscheme $C$ is $k$-Baire.
		\item\label{it:hypersurface} Assume that $n:=\dim X-1$ is finite. If $X$ has uncountably many (pairwise set-theoretically distinct) irreducible, $k$-Baire, closed subschemes of dimension $n$, then $X$ is $k$-Baire. \end{enumerate}
\end{lm}
\begin{proof} 	\begin{enumerate}\item Assume that there is such a component $C$.  Consider a sequence of closed subschemes $\{Z_i\}_{i\ge1}$ of $X$ with $\dim Z_i<\dim X$ for all $i\ge1$. Then for every $i\ge1$, one has $\dim C\cap Z_i\le \dim Z_i<\dim X=\dim C$.  Since $C$ is $k$-Baire, the set $C(k)\setminus \cup_i (C\cap Z_i)(k)$ is uncountable. Therefore, $X(k)\setminus\cup_iZ_i(k)$ is also uncountable.

Conversely, assume that every component of $X$ of maximum dimension is not $k$-Baire. As	$X$ is Noetherian, one can write $X=\cup_{j=1}^nC_j$ as a finite union of the irreducible components.  For every $j$ with $\dim C_j=\dim X$, the scheme $C_j$ is not $k$-Baire.  Therefore, there is a sequence $\{Z^j_i\}_{i\ge 1}$ of closed subschemes of $C_j$ such that $\dim Z^j_i<\dim C_j$  for all $i$ and that $C_j(k)\setminus\cup_iZ^j_i(k)$ is  countable. The finite family of components $C_k$ with $\dim C_k<\dim X$, joint with the sequences $\{Z_i^j\}_i$ for all $j$ with $\dim C_j=\dim X$, gives
a countable family $\{Z_s\}_s$ of closed subschemes of $X$ with $\dim Z_s<\dim X$ for all $s$. Then $X(k)\setminus(\cup_sZ_s(k))$ is countable, so $X$ is not $k$-Baire. 

\item Consider a sequence of closed subschemes $\{Z_i\}_{i\ge1}$ of $X$ with $\dim Z_i<\dim X$ for all $i\ge1$. Every $Z_i$ is a Noetherian scheme, so it has only finitely many irreducible components. The set of irreducible components of the family $\{Z_i\}_i$ is  countable. Thus, one may assume that every $Z_i$ is irreducible. By assumption, $X$ has an $n$-dimensional, irreducible, $k$-Baire closed subscheme $X'$ which is set-theoretically distinct from any $Z_i$. For every $i\ge1$, because $\dim X'=n\ge \dim Z_i$ and $Z_i$ is irreducible, one has $X'\not\subset Z_i$ and $X'\cap Z_i\neq X'$. Since $X'$ is irreducible, one has $\dim(X'\cap Z_i)<\dim X'$. As $X'$ is $k$-Baire, the set $X'(k)\setminus\cup_{i\ge 1}(X'\times_XZ_i)(k)$ is uncountable. It is a subset of $X(k)\setminus\cup_{i\ge 1}Z_i(k)$. Therefore, $X$ is $k$-Baire.\end{enumerate}
\end{proof}
Lemma \ref{lm:Qing} is well-known.
\begin{lm}\label{lm:Qing}
	If $X$  is a finite type scheme over $k$ with $\dim X>0$, then $X$ is $k$-Baire.
\end{lm}
\begin{proof}
	Since $X$ is of finite type over $k$, its dimension $m$ is finite and $X$ has only finitely many irreducible components. Replacing $X$ with an irreducible component of dimension $m$, one may that assume $X$ is irreducible. Then by \cite[Exercise 3.20 (e), p.94]{hartshorne2013algebraic}, every nonempty open subset of $X$ has dimension $m$. Replacing $X$ by an affine open, one may assume that $X$ is affine.  	By Noether's normalization lemma, there is a finite surjective morphism $p:X\to \mathbf{A}^m_k$ over $k$. By Lemma \ref{lm:upBaire},	 one may assume $X=\mathbf{A}^m_k$.

	By induction on $m>0$, we prove that $\mathbf{A}_k^m$ is $k$-Baire. When $m=1$, one has $\dim \mathbf{A}^1_k=1$, and $\mathbf{A}^1_k(k)$ is uncountable. By Remark \ref{rk:Bairebasic},  $\mathbf{A}^1_k$ is $k$-Baire. Assume the statement  for $m-1$ with $m\ge 2$. 	The set of hyperplanes in $\mathbf{A}^m_k$ is uncountable. By the inductive hypothesis, every hyperplane is $k$-Baire. From Lemma \ref{lm:Bairecomp} \ref{it:hypersurface}, so is $\mathbf{A}^m_k$. The induction is completed.
\end{proof}\subsubsection*{Baireness of cotori}
We  show that every positive dimensional cotorus  is $\bar{\Q}_{\ell}$-Baire.  
\begin{df}[{\cite[Def.~1, p.205]{bosch1984non}}]
Let $A$ be a $k$-algebra, and let $A[X]\to B$ be an injective ring map. We say that $B$ is $k$-\emph{Rückert} over $A$ if there is a nonempty family $W$ of \emph{monic} polynomials in $A[X]$ such that the following axioms are fulfilled.
\begin{enumerate}
	\item\label{it:Wprod} If $f,g\in A[X]$ are monic polynomials with $fg\in W$, then $f,g\in W$.
	\item\label{it:modw} For every $w\in W$, the $A$-algebra $B/w $ is isomorphic to  $A[X]/w $.
	\item\label{it:aut} For every $b\in B\setminus\{0\}$, there is an automorphism $\sigma$ of the $k$-algebra $B$  and a unit $u\in B^{\times}$ satisfying $u\sigma(b)\in W$. 
\end{enumerate}
\end{df}\begin{rk}\label{rk:W=1}From Axiom \ref{it:Wprod}, one gets $1\in W$. If $W=\{1\}$, then by Axiom \ref{it:aut}, for every $b\in B\setminus\{0\}$, one has $b\in B^{\times}$, \textit{i.e.},  $B$ is a field. Conversely, if $B$ is a field, then $B$ is $k$-Rückert over $A$ with $W=\{1\}$.

If $W\neq\{1\}$, then $\Spec(B)\to \Spec(A)$ is surjective. Indeed, take $w(\neq 1)\in W$.  By Axiom \ref{it:modw}, there is an $A$-isomorphism $B/w\to A[X]/w$, hence an isomorphism $\Spec(A[X]/w)\to \Spec(B/w)$ of schemes over $A$. Because $w$ is a monic polynomial different from $1$,  the ring map $A\to A[X]/w$ is injective and finite. The induced morphism $\Spec(A[X]/w)\to \Spec(A)$ is surjective, so $\Spec(B/w)\to \Spec(A)$  is surjective.  \end{rk}
For a commutative ring $R$ and an ideal $I\subset R$, let $V_R(I)=\Spec R/I(\subset \Spec R)$. For  $r\in R$, let $V_R(r)=V_R(rR)$.  Lemma \ref{lm:lift} is used in the induction step of the proof of Lemma \ref{lm:XnBaire}.
\begin{lm}\label{lm:lift}
Let  $A$ be Noetherian $k$-algebra of  dimension $n$. Let $B$ be a domain, but not a field, containing $A[X]$. Assume that $B$ is $k$-Rückert over $A$. 
\begin{enumerate}
	\item\label{it:liftdim} The ring $B$ is Noetherian of  dimension $n+1$.
	\item\label{it:liftBaire} Suppose  that $A$ is $k$-Baire. Let $S$ be an uncountable subset of $A$ such that for every $s\in S$, one has $\dim V_A(s)=n-1$. Suppose that the family $\{V_A(s)\}_{s\in S}$  is pairwise disjoint. Then $B$ is   $k$-Baire.
\end{enumerate}
\end{lm}
\begin{proof}By Axiom \ref{it:aut}, for every $b\in B\setminus(B^{\times}\cup\{0\})$,  there is an automorphism $\sigma$ of the $k$-algebra  $B$  and a unit $u\in B^{\times}$ such that $w:=u\sigma(b)$ is in $W$. Since $b$ is not a unit, one has $w\neq1$. By Axiom \ref{it:modw}, the $A$-algebra $B/w$ is isomorphic to $A[X]/w$. Since $w(\neq1)$ is a monic polynomial over $A$, the ring map $A\to A[X]/w$ is injective finite.  
\begin{enumerate}
	\item  One has 	\begin{equation}\label{eq:dimB/b}\dim B/b= \dim B/w=\dim A[X]/w\labelrel={myeq:00ok}\dim A=n,\end{equation}where \eqref{myeq:00ok} uses  \cite[\href{https://stacks.math.columbia.edu/tag/00OK}{Tag 00OK}]{stacks-project}.
	The domain $B$ is not a field, so $\dim B=n+1$. By \cite[Prop.~2, p.206]{bosch1984non}, the ring $B$ is Noetherian. 
	
	\item The morphism $\Spec A[x]/w\to \Spec A$	is finite surjective. Then by Lemma \ref{lm:upBaire}, the algebra $A[X]/w$ is $k$-Baire. 
	As $\sigma$ is over $k$, the $k$-algebra $B/b$ is isomorphic to $B/w$. Then $B/b$ is  $k$-Baire. 
	
	For every $s\in S$, one has $\dim V_A(s)<\dim A$, so $s\neq0$.
	 From Remark \ref{rk:W=1}, as $B$ is not a field, the morphism $\Spec(B)\to \Spec(A)$ is surjective.	The preimage of $V_A(s)$ under the surjection $\Spec(B)\to \Spec(A)$ is $V_B(s)$, so  $V_B(s)$ is nonempty. In particular, $s\notin B^{\times}$ and $B/s$ is $k$-Baire.  Moreover, the family $\{V_B(s)\}_{s\in S}$ is pairwise disjoint. By (\ref{eq:dimB/b}), one gets $\dim V_B(s)=n$.
	
	By Part \ref{it:liftdim}, $B$ is Noetherian. Then by Lemma \ref{lm:Bairecomp} \ref{it:component}, for every $s\in S$,  there is a $k$-Baire irreducible component $C_s\subset \Spec(B/s)$ of dimension $n$. The family  $\{C_s\}_{s\in S}$ is pairwise disjoint. From Lemma \ref{lm:Bairecomp} \ref{it:hypersurface}, $B$ is $k$-Baire.
\end{enumerate}
\end{proof}

Set $S_n:=\bar{\Q}_{\ell}\otimes_{\Z_{\ell}}\Z_{\ell}[[X_1,\dots,X_n]])$. By \cite[Prop.~3.2.2 (1)]{gabber1996faisceaux}, the natural morphism $S_n\to \bar{\Q}_{\ell}[[X_1,\dots,X_n]]$ is injective. Then the isomorphism (\ref{eq:Iwasawaiso}) identifies $R_n$  with the $\bar{\Q}_{\ell}$-subalgebra $S_n\subset \bar{\Q}_{\ell}[[X_1,\dots,X_n]]$. 
\begin{ft}\label{ft:AnAn-1}
For every integer $n\ge0$,
\begin{enumerate}
	\item\label{it:Andim} \textup{(\cite[Thm.~A.2.1, Prop,~A.2.2.1]{gabber1996faisceaux})}  the ring	$S_n$ is  a Noetherian, regular, Jacobson domain of Krull dimension $n$;
	\item \textup{(\cite[Prop A.2.2.2, proof of A.2.2.3 (ii)]{gabber1996faisceaux})}  $S_{n+1}$ is $\bar{\Q}_{\ell}$-Rückert over $S_n$.
\end{enumerate} 
\end{ft}

\begin{lm}\label{lm:XnBaire}
For every integer $n\ge 1$, the algebra $S_n$ is $\bar{\Q}_{\ell}$-Baire.
\end{lm}
\begin{proof}
Since $\bar{\Q}_{\ell}$ is a flat $\Z_{\ell}$-module, the injection $\Z_{\ell}[X_1,\dots,X_n]\to \Z_{\ell}[[X_1,\dots,X_n]]$ induces an injection $\bar{\Q}_{\ell}[X_1,\dots,X_n]\to S_n$. The natural morphism \begin{equation}\label{eq:X1A1}
	\Spec(\bar{\Q}_{\ell}[[X_1,\dots,X_n]])\to \mathbf{A}^n_{\bar{\Q}_{\ell}}
\end{equation} of schemes over $\bar{\Q}_{\ell}$
factors through a morphism $p_n:\Spec(S_n)\to \mathbf{A}^n_{\bar{\Q}_{\ell}}$.

Let $\cM=\cup_E m_E$, where $E$ runs through all finite subextensions of $\Q_{\ell}\subset\bar{\Q}_{\ell}$, and $m_E$ is the maximal ideal of the ring of integers of $E$.  Then $\cM$ is the maximal ideal of the integral closure $\overline{\Z_{\ell}}$ of $\Z_{\ell}$ inside $\bar{\Q}_{\ell}$. By \cite[Prop., p.128]{robert2000course}, the residue field $\overline{\Z_{\ell}}/\cM$ is an algebraic closure of the finite field $F_{\ell}$, so it is countable. As $\Z_{\ell}$ is uncountable, so is the set $\cM$. 

For every $(a_1,\dots,a_n)\in \cM^n$,  there is a surjective morphism of $\bar{\Q}_{\ell}$-algebras: \[ \bar{\Q}_{\ell}[[X_1,\dots,X_n]]\to \bar{\Q}_{\ell},\quad f\mapsto f(a_1,\dots,a_n).\] Its kernel is a $\bar{\Q}_{\ell}$-point of $\Spec(\bar{\Q}_{\ell}[[X_1,\dots,X_n]])$, whose image under   (\ref{eq:X1A1}) is $(a_1,\dots,a_n)\in A^n_{\bar{\Q}_{\ell}}(\bar{\Q}_{\ell})$. Hence $\cM^n\subset p_n(\Spec(S_n)(\bar{\Q}_{\ell}))$. In particular, $\Spec(S_n)(\bar{\Q}_{\ell})$ is uncountable.
By induction on $n>0$, we prove that $S_n$ is $\bar{\Q}_{\ell}$-Baire, and $\{V_{S_n}(X_1-a)\}_{a\in \cM}$ is a pairwise disjoint family of $(n-1)$-dimensional subsets. When $n=1$, by Fact \ref{ft:AnAn-1} \ref{it:Andim}, $S_1$ is $\bar{\Q}_{\ell}$-Baire. Moreover, $\{V_{S_1}(X_1-a)\}_{a\in\cM}$ is a pairwise distinct family of closed point of $\Spec(S_1)$.  The statement is proved for $n=1$. Assume the statement  for $n-1$ with $n\ge 2$. By Fact \ref{ft:AnAn-1}, \eqref{eq:dimB/b},  and Lemma \ref{lm:lift} \ref{it:liftBaire}, the statement holds for $n$. The induction is completed.
\end{proof}

	\section{Krämer-Weissauer theory}\label{sec:KWtheory}
Let $k$ be a field of characteristic zero. Let $\Vec_k$ be the category of finite dimensional $k$-vector spaces. Choose an algebraic closure $\bar{k}$ of $k$. Let $\Rep_{\bar{\Q}_{\ell}}(\Gamma_k)$ be the category of continuous, finite dimensional $\bar{\Q}_{\ell}$-representations of $\Gamma_k$.  Let $A$ be an abelian variety over $k$. Recall that $\pet(A_{\bar{k}})$ is a free $\hat{\Z}$-module of rank $2\dim A$. With the notation of Section \ref{sec:cotori}, set
\begin{itemize}
\item $\cC(A)=\cC(\pet(A_{\bar{k}}))$: the group of characters $\pet(A_{\bar{k}})\to \bar{\Q}_{\ell}^{\times}$;
\item $\cC(A)_{\ell'}=\cC(\pet(A_{\bar{k}}))_{\ell'}$: the group of characters of finite order prime to $\ell$;
\item $\cC(A)_{\ell}$: the cotorus assigned to $\pet(A_{\bar{k}})$. 
\end{itemize} 

\subsection{Generic vanishing theorem}
For an object $K\in \Perv(A)$,   set \[\cS(K):=\{\chi\in \cC(A)|H^i(A_{\bar{k}},K\otimes^L L_{\chi})\neq0\text{ for some integer }i\neq0\}.\] %

\begin{ft}[{\cite[Thm.~1.1]{kramer2015vanish}, \cite[Vanishing Theorem, p.561; Thm.~2]{weissauer2016vanishing}}]\label{ft:etvanish}
For every perverse sheaf $K\in \Perv(A)$ and every character $\chi_{\ell'}\in\cC(A)_{\ell'}$, the set \[\{\chi_{\ell}\in\cC(A)_{\ell}(\bar{\Q}_{\ell})|\chi_{\ell'}\chi_{\ell}\in \cS(K)\}\] is the set of $\bar{\Q}_{\ell}$-points of a strict Zariski closed subset of the scheme $\cC(A)_{\ell}$.\end{ft}

We review \cite[p.725]{kramer2015tannaka}. Because of $\mathrm{char}(k)=0$, for every $K\in \Perv(A)$, its Euler characteristic $\chi(A,K):=\sum_{i\in \Z}(-1)^i\dim_{\bar{\Q}_{\ell}}H^i(A_{\bar{k}},K)$ satisfies \begin{equation}\label{eq:Eulerineq}\chi(A,K)\ge0.\end{equation} Let $N(A)\subset \Perv(A)$ be the full subcategory of objects $K$ with $\chi(A,K)=0$. From  the additivity of the function $\chi(A,-):\mathrm{Ob}(\Perv(A))\to \N$ and \eqref{eq:Eulerineq}, $N(A)$ is a Serre subcategory of $\Perv(A)$. Let $\bar{P}(A):=\Perv(A)/N(A)$ be the quotient abelian category.
For every $\chi\in\cC(A)$,  set  $\cE^{\chi}(A_{\bar{k}})=\{K\in \Perv(A_{\bar{k}})|\chi\notin \cS(K)\}$. Then $\cE^{\chi}(A_{\bar{k}})$ is closed under extensions in $\Perv(A_{\bar{k}})$. Let $P^{\chi}(A)\subset\Perv(A)$ be the full subcategory of objects $K$ with $Q\in\cE^{\chi}(A_{\bar{k}})$ for every simple subquotient $Q$ of $K|_{A_{\bar{k}}}$ in $\Perv(A_{\bar{k}})$. 
By \cite[Thm.~4.3.1 (i)]{beilinson2018faisceaux}, every object $K\in\Perv(A)$ is Noetherian and Artinian. Then by  Fact \ref{ft:etvanish} and Lemma \ref{lm:finisubq} \ref{it:subq=factor}, for  every  $\chi_{\ell'}\in\cC(A)_{\ell'}$,  the set $\{\chi_{\ell}\in\cC(A)_{\ell}(\bar{\Q}_{\ell})|K\in P^{\chi_{\ell'}\chi_{\ell}}(A)\}$ is the set of $\bar{\Q}_{\ell}$-points of a strict Zariski closed subset of $\cC(A)_{\ell}$. 

\begin{lm}\label{lm:AES}
Let $\cA$ be a Noetherian and Artinian abelian category. Let $\cE$ be a class of objects of $\cA$ closed under isomorphisms. Let $\cS\subset\cA$ be the full subcategory of objects  whose all  simple subquotients are in $\cE$. 
\begin{enumerate}
\item\label{it:SE} Then $\cS$ is a Serre subcategory of $\cA$.
\item\label{it:SinE} If further $\cE$ is closed under extensions,  then $\cS\subset \cE$.\end{enumerate} 
\end{lm}
\begin{proof}\hfill\begin{enumerate}
		\item \begin{enumerate}\item\label{it:Sclssubquotient} We prove that $\cS$ is closed under subquotients. Let $X$ be an object of $\cS$ with a subquotient $Y$. Every simple subquotient of $Y$ is that of $X$, hence in $\cE$. Thus, $Y\in\cS$.  \end{enumerate}
		
		Let $0\to L\overset{f}{\to} M\overset{g}{\to} N\to0$ be a short exact sequence in $\cA$ with $L,N\in\cS$. 	Let $Q$ be a  simple subquotient of $M$. We prove  $Q\in\cE$.
			\begin{enumerate}\setItemnumber{2}\item\label{it:quotientofextension} First, assume that $Q$ is a quotient of $M$. The natural morphism $L\to Q$ is either an epimorphism or zero, in which case $Q$ is a simple quotient of $L$ or $N$ respectively. Hence $Q\in\cE$. 
		 \setItemnumber{3}
	\item\label{it:subquotientofextension}	Now assume that $Q$ is general. There is a subobject $M_0\subset M$ and an epimorphism $M_0\to Q$. Then \[0\to f^{-1}(M_0)\to M_0\to g(M_0)\to0\] is a short exact sequence in $\cA$ with $f^{-1}(M_0)$ (resp. $g(M_0)$) a subobject of $L$ (resp. $N$). From Part \ref{it:Sclssubquotient}, both $f^{-1}(M_0)$ and $g(M_0)$ are in $\cS$. From Part \ref{it:quotientofextension}, one has $Q\in\cE$. \end{enumerate}
		
		From Part \ref{it:subquotientofextension}, one has $M\in\cS$, and $\cS$ is closed under extensions. The result follows from \cite[\href{https://stacks.math.columbia.edu/tag/02MP}{Tag 02MP}]{stacks-project}.
\item By \cite[\href{https://stacks.math.columbia.edu/tag/0FCJ}{Tag 0FCJ}]{stacks-project}, every object $X\in\cS$ admits a filtration in $\cA$ \[0\subset X_1\subset X_2\subset\dots\subset X_n=X\] by subobjects such that each $X_i/X_{i-1}$ is a simple subquotient of $X$. Then $X_i/X_{i-1}\in \cE$. As $\cE$ is closed under extensions, one has $X\in\cE$.\end{enumerate}
\end{proof}
By Lemma \ref{lm:AES} \ref{it:SE}, for every $\chi\in\cC(A)$,  the subcategory $P^{\chi}(A)\subset \Perv(A)$ is a Serre subcategory. From Lemma \ref{lm:AES} \ref{it:SinE}, for every $K\in P^{\chi}(A)$ and every  integer $i\neq0$, one has \begin{equation}\label{eq:vanishPchi}H^i(A_{\bar{k}},K\otimes^LL_{\chi})=0.\end{equation} From the proof of \cite[Lemma~2.5 (3)]{lawrence2020shafarevich}, the functor \begin{equation}\label{eq:omegachi}\omega_{\chi}=H^0(A_{\bar{k}},\cdot\otimes^LL_{\chi}):P^{\chi}(A)\to \Vec_{\bar{\Q}_{\ell}} \end{equation} is exact. Let $N^{\chi}(A)$ be the full subcategory of $P^{\chi}(A)$ of objects in $N(A)$. By \cite[Cor.~4.2]{kramer2015vanish}, for every $K\in N^{\chi}(A)$, , one has $\chi(A,K\otimes^LL_{\chi})=0$. From (\ref{eq:vanishPchi}), one has $H^0(A_{\bar{k}},K\otimes^LL_{\chi})=0$. Then by \cite[\href{https://stacks.math.columbia.edu/tag/02MS}{Tag 02MS}]{stacks-project}, the functor $\omega_{\chi}$ factors uniquely through an exact functor (still denoted by $\omega_{\chi}$) \begin{equation}\label{eq:quotientomegachi}P^{\chi}(A)/N^{\chi}(A)\to \Vec_{\bar{\Q}_{\ell}}.\end{equation} 

\subsection{Tannakian groups}
Let  $(\cC,\otimes)$ a neutral Tannakian category (in the sense of \cite[Def.~2.19]{deligne2022tannakian}) over an algebraically closed field $Q$ of characteristic $0$, with a fiber functor $\omega:\cC\to \Vec_Q$. Let $\Aut^{\otimes}(\cC,\omega)$ be the corresponding affine group scheme over $Q$. By \cite[Sec.~9.2, p.187]{deligne1990categories}, up to isomorphism of group schemes, $\Aut^{\otimes}(\cC,\omega)$ is independent of the choice of $\omega$. (See \cite[Thm.~1.2]{wibmer2022remark} for an elementary proof.)

For an object $K\in \cC$, let  $\iota:\langle K\rangle\hookrightarrow \cC$ be the full subcategory whose objects are the subquotients of $\{(K\oplus K^{\vee})^{\otimes n}\}_{n\ge1}$. Then $(\langle K\rangle,\otimes)$ is a neutral Tannakian subcategory of $\cC$ (in the sense of \cite[1.7]{milne2005quotients}), for which  $\omega\circ\iota:\langle K\rangle\to \Vec_Q$ is  a fiber functor. The group scheme $\Aut^{\otimes}(\langle K\rangle,\omega\circ\iota)$ is the image of the natural morphism $\Aut^{\otimes}(\cC,\omega) \to \GL(\omega(K))$. 
\begin{df}\label{df:Tannakianmonodromy} The algebraic group $\Aut^{\otimes}(\langle K\rangle,\omega\circ\iota)$  is called the Tannakian monodromy group\index{Tannakian monodromy group} of $K$ at $\omega$ and is denoted by $G_{\omega}(K)$.\end{df}By \cite[p.69]{simpson1992higgs}, $G_{\omega}(K)$ is reductive if and only if $K$ is semisimple in $\cC$. 

\begin{eg}With  tensor product, $\Rep_{\bar{\Q}_{\ell}}(\Gamma_k)$ is a neutral Tannakian category over $\bar{\Q}_{\ell}$.  The forgetful  functor $\omega:\Rep_{\bar{\Q}_{\ell}}(\Gamma_k)\to \Vec_{\bar{\Q}_{\ell}}$ is a fiber functor. The Tannakian monodromy group of an object $\rho:\Gamma_k\to \GL(V)$ at $\omega$ is the Zariski closure of $\rho(\Gamma_k)$ in $\GL(V)$. \end{eg}
\subsection{Sheaf convolution}
 Let $m:A\times_kA\to A$ be the group law on $A$. Let $p_i:A\times_kA\to A $ be the projection to $i$-th factor ($i=1,2$).  The  bifunctor \[*:D_c^b(A)\times D_c^b(A)\to D_c^b(A),\quad -*+:=Rm_*(p_1^*-\otimes^Lp_2^*+)\] is called the \emph{convolution} on $A$. \begin{eg}
	For every closed reduced subvariety $i:X\to A$, let $\delta_X:=i_*\bar{\Q}_{\ell,X}\in D_c^b(A)$. Then for every closed point $x\in A$, one has $\delta_x*\delta_X=\delta_{x+X}$.
\end{eg}

By \cite{weissauer2011remark} and \cite[Sec.~3.1]{javanpeykar2023monodromy}, the pair $(D_c^b(A),*)$ is a rigid, symmetric monoidal category, with unit $\delta_0$. For every $K\in D_c^b(A)$, its adjoint dual is $K^{\vee}:=[-1]_A^*\D_AK$. 
\begin{ft}[{\cite[proof of Thm.~13.2]{kramer2015vanish}, \cite[Lemma~2.5 (4)]{lawrence2020shafarevich}, \cite[Prop.~3.1]{javanpeykar2023monodromy}}]
The convolution on $A$ induces a bifunctor \[\bar{P}(A)\times \bar{P}(A)\to \bar{P}(A),\quad (-,+)\mapsto{}^p\cH^0(-*+)\] fitting into a commutative square \begin{center}\begin{tikzcd}
	\Perv(A)\times \Perv(A) \arrow[r, "*"] \arrow[d] & D_c^b(A) \arrow[d, "{}^p\cH^0"] \\
	\bar{P}(A)\times \bar{P}(A) \arrow[r, dashed]    & \bar{P}(A).                     
\end{tikzcd}\end{center}
It makes $\bar{P}(A)$ a neutral Tannakian category  over $\bar{\Q}_{\ell}$. For every $\chi\in\cC(A)$, the subcategory $P^{\chi}(A)/N^{\chi}(A)\subset\bar{P}(A)$ is a Tannakian subcategory, on which (\ref{eq:quotientomegachi}) is a fiber functor.
\end{ft}
\begin{eg}\cite[Example 7.1]{kramer2015tannaka}\label{eg:skyscraper}
	Fix a closed point $x\in  A$. Then $\delta_x$ is a simple object of $\Perv(A)$. As $\cS(\delta_x)$ is empty, for every $\chi\in \cC(A)$, one has $\delta_x\in P^{\chi}(A)$. If $x$ is a torsion point of order $n$, then $G_{\omega_{\chi}}(\delta_x)$ is isomorphic to $\Z/n$. If $x$ is not a torsion point, then $G_{\omega_{\chi}}(\delta_x)$ is isomorphic to $\G_{m/\bar{\Q}_{\ell}}$. \end{eg}
Let $\psi:\pet(A)\to \bar{\Q}_{\ell}^{\times}$ be a character, and set $\psi'=\psi|_{\pet(A_{\bar{k}})}$.  The functor \[\omega_{\psi}:\Perv(A)\to \Rep_{\bar{\Q}_{\ell}}(\Gamma_k),\quad K\mapsto H^0(A_{\bar{k}},K\otimes^LL_{\psi})\] fits into a commutative square \begin{center}
	\begin{tikzcd}
		\Perv(A) \arrow[r, "\omega_{\psi}"]                    & \Rep_{\bar{\Q}_{\ell}}(\Gamma_k) \arrow[d, "\omega"] \\
		P^{\psi'}(A) \arrow[u, hook] \arrow[r, "\eqref{eq:omegachi}"] & \Vec_{\bar{\Q}_{\ell}}                              
	\end{tikzcd}
\end{center}
The  quotient functor $P^{\psi'}(A)/N^{\psi'}(A)\to  \Rep_{\bar{\Q}_{\ell}}(\Gamma_k)$ of $\omega_{\psi}|_{P^{\psi'}(A)}$ induces a morphism of affine groups schemes \begin{equation}\label{eq:omegapsiback}\omega_{\psi}^*:\Aut^{\otimes}(\Rep_{\bar{\Q}_{\ell}}(\Gamma_k),\omega)\to \Aut^*(P^{\psi'}(A)/N^{\psi'}(A),\omega_{\psi'}).\end{equation}

\begin{df}\label{df:Gmon}For every $K\in \Perv(A)$, let $\Mon(K,\psi)$ be the Tannakian monodromy group\index{$\Mon(K,\psi)$ monodromy group} of $\omega_{\psi}(K)$ in $\Rep_{\bar{\Q}_{\ell}}(\Gamma_k)$ at $\omega$.\end{df} For every $K\in P^{\psi'}(A)$, the  functor $\omega_{\psi}|_{\langle K\rangle}:\langle K\rangle\to \langle\omega_{\psi}(K)\rangle$ induces a closed immersion of linear algebraic groups $\omega_{\psi}^*:\Mon(K,\psi)\to G_{\omega_{\psi'}}(K)$, which is the projection of (\ref{eq:omegapsiback}) in $\GL(\omega_{\psi'}(K))$.

	\section{Main results}
Consider Setting \ref{set:AX}. For every character $\chi\in \cC(A)$, denote the pullback of $\chi$ along $(p_A|_{A_{\eta}})_*:\pet(A_{\eta})\to \pet(A)$ by $\chi_{\eta}:\pet(A_{\eta})\to \bar{\Q}_{\ell}^{\times}$. Then the restriction $\chi_{\eta}|_{\pet(A_{\bar{\eta}})}$ is identified with $\chi$ \textit{via} the isomorphism $(p_A|_{A_{\bar{\eta}}})_*:\pet(A_{\bar{\eta}})\to \pet(A)$.   \begin{rk}\label{rk:monogp}
Let $K\in D^{\ULA}(A\times X/X)$. For every character $\chi\in \cC(A)$, since $L_{\chi}$ is a lisse sheaf of rank $1$, one has $K\otimes p_A^*L_{\chi}\in D^{\ULA}(A\times X/X)$. By Fact \ref{ft:ULA} \ref{it:properULA}, as $p_X:A\times X\to X$ is proper, one has $Rp_{X*}(K\otimes p_A^*L_{\chi})\in D^{\ULA}(X/X)$. Then from Fact \ref{ft:ULA} \ref{it:ULAoverid}, for every integer $n$, $L^n(K,\chi):=\cH^nRp_{X*}(K\otimes p_A^*L_{\chi})$ is a lisse sheaf on $X$. 	By the proper base change theorem (see, \textit{e.g.}, \cite[\href{https://stacks.math.columbia.edu/tag/095T}{Tag 095T}]{stacks-project}), one has \[L^n(K,\chi)_{\bar{\eta}}=H^n(A_{\bar{\eta}},K|_{A_{\bar{\eta}}}\otimes^LL_{\chi}).\] Suppose that $X$ is normal. Then the natural morphism $\eta_*:\Gamma_{k(\eta)}\to \pet(X,\bar{\eta})$ is surjective. Its composition with the monodromy representation \[\pet(X,\bar{\eta})\to \GL(L^n(K,\chi)_{\bar\eta})\] is the natural Galois representation
$\Gamma_{k(\eta)}\to \GL\left( H^n(A_{\bar{\eta}},K|_{A_{\eta}}\otimes^LL_{\chi})  \right)$. Therefore, the monodromy group of $H^n(A_{\bar{\eta}},K|_{A_{\bar{\eta}}}\otimes^LL_{\chi})$ in $\Rep(\Gamma_{k(\eta)})$ coincides with that of $L^n(K,\chi)$ in $\Loc(X)$.\end{rk}
Let $K\in\Perv(A\times X/X)$. Write $\Mon(K,\chi)$ for $\Mon(K|_{A_{\eta}},\chi_{\eta})$. By Remark \ref{rk:monogp}, when $K$ is ULA and $X$ is normal, $\Mon(K,\chi)$ is the monodromy group of the lisse sheaf $L^0(K,\chi)$. We shall prove that there exist many characters $\chi$ with the monodromy group $\Mon(K,\chi)$  normal in the convolution group $G(K|_{A_{\eta}})$. By the well-known normality criterion (Lemma \ref{lm:normalcriterion}), it suffices to show that $\Mon(K,\chi)$ is reductive, and to consider the monodromy fixed part of \emph{all} the representations of $G(K|_{A_{\eta}})$. Such representations are from perverse sheaves. 
\begin{lm}\label{lm:normalcriterion}
Let $G$ be a linear algebraic group over an algebraically closed field $C$. Let $H$ be a closed, reductive subgroup of $G$. If for every $V\in \Rep_C(G)$, the  subspace $V^H$ is $G$-stable, then $H$ is normal in $G$.
\end{lm}
\begin{proof}
By \cite[Cor.~2.4]{grosshans2006algebraic} and reductivity, $H$ is observable in $G$ (in the sense of \cite[p.134]{bialynicki1963extensions}). From  \cite[Prop.~C.3]{andre20normality}, $H$ is normal in $G$. 
\end{proof}
\subsection{Reductivity}
\begin{lm}\label{lm:monreductive}Let $K\in\Perv(A\times X/X)$ be semisimple $D_c^b(A\times X)$.
For every   $\chi\in\cC(A)\setminus \cS(K|_{A_{\eta}})$, the monodromy group $\Mon(K,\chi)$ is reductive. 
\end{lm}
\begin{proof}
	By Lemma \ref{lm:shriksimple}, when $X$ is replaced by a nonempty open subset, the semisimplicity of $K$ in $D_c^b(A\times X)$ is preserved. Moreover,  the $\Gamma_{k(\eta)}$-representation $\omega_{\chi_{\eta}}(K|_{A_{\eta}})$ and hence the group $\Mon(K,\chi)$ remain unchanged. Thus, by \cite[\href{https://stacks.math.columbia.edu/tag/056V}{Tag 056V}]{stacks-project}, one may assume that $X$  is smooth. From Lemma \ref{lm:tensorLchi}, as $K$ is semisimple in $D_c^b(A\times X)$, so is $K\otimes^Lp_A^*L_{\chi}$. By Fact \ref{ft:BBD} \ref{it:decomp}, the object $Rp_{X*}(K\otimes^Lp_A^*L_{\chi})$ is semisimple in $D_c^b(X)$. 
	
 Since $\chi\notin \cS(K|_{A_{\eta}})$, for every integer $n\neq0$, one has $H^n(A_{\bar{\eta}},K|_{A_{\bar{\eta}}}\otimes^LL_{\chi})=0.$ By Fact \ref{ft:genericlisse}, there is a nonempty open subset $U_0$ (resp. $U_n$ for every integer $n\neq0$) of $X$ such that $[\cH^0Rp_{X*}(K\otimes^Lp_A^*L_{\chi})]|_{U_0}$ is a $\bar{\Q}_{\ell}$-lisse sheaf (resp. $[\cH^nRp_{X*}(K\otimes^Lp_A^*L_{\chi})]|_{U_n}=0$). The set \[J:=\{n\in\Z:\cH^nRp_{X*}(K\otimes^Lp_A^*L_{\chi})\neq0\}\] is finite and $X$ is irreducible, so $U:=U_0\cap\cap_{n\in J}U_n$ is a nonempty open subset of $X$. Shrinking $X$ to $U$, one may assume further that $\cH^nRp_{X*}(K\otimes^Lp_A^*L_{\chi})=0$ for every integer $n\neq0$, and that $\cH^0Rp_{X*}(K\otimes^Lp_A^*L_{\chi})$ is a $\bar{\Q}_{\ell}$-lisse sheaf on $X$. 
	
	Thus, the semisimple object $Rp_{X*}(K\otimes^Lp_A^*L_{\chi})[\dim X]$ of $D_c^b(X)$ lies in $\Perv(X)$, so it is semisimple in $\Perv(X)$. By \cite[Prop.~3.4.1]{achar2021perverse}, the object $Rp_{X*}(K\otimes^Lp_A^*L_{\chi})$ of $\Loc(X)$ is semisimple.  Because $X$ is smooth, the algebraic group $\Mon(K,\chi)$ is reductive.\end{proof}
\begin{eg}In Lemma \ref{lm:monreductive}, the identity component $\Mon(K,\chi)^0$ may not be semisimple.
	Let $X$ be a smooth, projective, integral algebraic curve over $k$ of genus $1$. Then $\pet(X,\bar{\eta})\cong\hat{\Z}^2$. There exists a character $\chi:\pet(X,\bar{\eta})\to \bar{\Q}_{\ell}^{\times}$ of \emph{infinite} order. Let $A=\Spec(k)$. Then $\cC(A)=\{1\}$ and $\Mon(L_{\chi}|,1)=\G_{m/\bar{\Q}_{\ell}}$ is an algebraic torus. 
\end{eg}

\begin{rk}\label{rk:fromsubvar}
Let $i:Y\to A\times X$ be a closed subvariety, such  that the induced morphism $f:Y\to X$ is smooth  with connected fibers of dimension $d$: \begin{center}
\begin{tikzcd}
	Y \arrow[r, "i", hook] \arrow[d, "f"] & A\times X \arrow[ld, "p_X"] \arrow[d, "p_A"] \\
	X                                     & A.                                            
\end{tikzcd}
\end{center}
By Example \ref{eg:subisULA}, one has $K:=i_*\bar{\Q}_{\ell,Y}[d]\in\Perv^{\ULA}(A\times X/X)$. By Fact \ref{ft:BBD} \ref{it:decomp}, it is semisimple in $D_c^b(A\times X)$. Assume that $X$ is smooth. Then for every $\chi\in \cC(A)\setminus \cS(K|_{A_{\eta}})$, the algebraic group $\Mon(K,\chi)$ coincides with the Zariski closure of the image of the  monodromy representation of the $\bar{\Q}_{\ell}$-lisse sheaf $R^df_*i^*p_A^*L_{\chi}$  on $X$, which is studied in  \cite[Sec.~1.5]{kramer2023arithmetic} (but with coefficient $\C$ instead of $\bar{\Q}_{\ell}$). 
\end{rk}
\subsection{Fixed part}Theorem \ref{thm:introfixed} follows from Theorem  \ref{thm:fixedpart} and Fact \ref{ft:etvanish}, because the union in Condition \ref{it:chinotHj} of Theorem \ref{thm:fixedpart} is in fact a finite union.
\begin{thm}\label{thm:fixedpart}
	Assume that $X$ is smooth. Let $K\in \Perv^{\ULA}(A\times X/X)$ be semisimple $D_c^b(A\times X)$. Then there exists a subobject $K^0\subset K$ in $\Perv^{\ULA}(A\times X/X)$ such that
	for every   $\chi\in \cC(A)$  with \begin{enumerate}
		\item\label{it:chinotHj}  $\chi\notin 	\cup_{j\in \Z}\cS({}^p\cH^j(Rp_{A*}K))$, 
				\item\label{it:chinotKeta} $K|_{A_{\eta}}\in P^{\chi}(A_{\eta})$ and
		\item\label{it:chinotH0} ${}^p\cH^0(Rp_{A*}K)\in P^{\chi}(A)$,
	\end{enumerate}
	one has $\omega_{\chi_{\eta}}(K^0|_{A_{\eta}})=\omega_{\chi_{\eta}}(K|_{A_{\eta}})^{\Gamma_{k(\eta)}}$.
\end{thm}
\begin{proof}
	
	By Remark \ref{rk:monogp}, as $K$ is ULA, $\cH^0Rp_{X*}(K\otimes p_A^*L_{\chi})$ is a lisse sheaf on $X$, and since $X$ is smooth,  the canonical morphism $\Gamma_{k(\eta)}\to \pet(X,\bar{\eta})$ is surjective. Thus, from Fact \ref{ft:BBD} \ref{it:globalinv}, as $K$ is semisimple in $D_c^b(A\times X)$, the natural map \begin{equation}
		\label{eq:global}H^0(A\times X,K\otimes^L p_A^*L_{\chi})\to \omega_{\chi_{\eta}}(K|_{A_{\eta}})^{\Gamma_{k(\eta)}}
	\end{equation}
	is surjective.
	
	By Fact \ref{ft:cstprojection}, one has \begin{equation}\label{eq:proj}H^0(A\times X,K\otimes^L p_A^*L_{\chi})=H^0(A,(Rp_{A*}K)\otimes^L L_{\chi}).\end{equation} By Condition \ref{it:chinotHj}, for  any integers $i\neq0$ and $j$,  one has \[H^i(A,{}^p\cH^j(Rp_{A*}K)\otimes^L L_{\chi})=0.\] By Lemma \ref{lm:tensorLchi}, the spectral sequence in \cite[Rk.~8.1.14 (6)]{maxim2019intersection} becomes \[E_2^{i,j}=H^i(A,{}^p\cH^j(Rp_{A*}K)\otimes^L L_{\chi})\Rightarrow H^{i+j}(A,(Rp_{A*}K)\otimes^LL_{\chi}).\] It degenerates at page $E_2$.  Hence \begin{equation}\label{eq:spectral}
		H^0(A,(Rp_{A*}K)\otimes^L L_{\chi})=H^0(A,({}^p\cH^0Rp_{A*}K)\otimes^L L_{\chi}).
	\end{equation}

	Set $K^1:=p_A^*{}^p\cH^0(Rp_{A*}K)\in D_c^b(A\times X)$. By Fact \ref{ft:ULA} \ref{it:ULAoverfld}, one has ${}^p\cH^0(Rp_{A*}K)\in D^{\ULA}(A/k)$. From Fact \ref{ft:ULA} \ref{it:basechgULA}, one gets $K^1\in D^{\ULA}(A\times X/X)$.
	For every $x\in X(k)$,  the restriction $p_A|_{A_x}:A_x\to A$ is an isomorphism of abelian varieties over $k$, so  the  functor 
		$(p_A|_{A_x})^*:	\Perv(A)\to \Perv(A_x)$	is an equivalence of abelian categories.  It sends 	${}^p\cH^0(Rp_{A*}K)$ to $K^1|_{A_x}$, so $K^1|_{A_x}\in \Perv(A_x)$ and hence $K^1\in \Perv^{\ULA}(A\times X/X)$. From $K^1|_{A_{\eta}}=(p_A|_{A_{\eta}})^*{}^p\cH^0(Rp_{A*}K)$ and Condition \ref{it:chinotH0}, one has $K^1|_{A_{\eta}}\in P^{\chi}(A_{\eta})$. Then \begin{equation}\label{eq:omegaK1}\omega_{\chi}(K^1|_{A_{\eta}})=H^0(A,{}^p\cH^0(Rp_{A*}K)\otimes^L L_{\chi}).\end{equation}

By \cite[4.2.4]{beilinson2018faisceaux}, as	every fiber of  $p_A:A\times X\to A$ has dimension $\dim X$,  the functor \[Rp_{A*}[-\dim X]:D_c^b(A\times X)\to D_c^b(A)\] is left t-exact for the absolute perverse t-structures. From Lemma \ref{lm:shift}, as $X$ is smooth and $K$ is ULA, one has $K[\dim X]\in\Perv(A\times X)$	 and so $Rp_{A*}K\in {}^pD^{\ge0}(A)$.  Taking the perverse truncation, one has ${}^p\tau^{\le0}(Rp_{A*}K)={}^p\cH^0(Rp_{A*}K)$. \textit{Via} the adjunction formula (see, \textit{e.g.}, \cite[p.107]{kiehl2001weil}),  the natural morphism \[{}^p\tau^{\le0}(Rp_{A*}K)\to Rp_{A*}K\] in $D_c^b(A)$ (from the definition of t-structure) induces a morphism $h:K^1\to K$ in $D_c^b(A\times X)$.  Then $h$ is a morphism in $\Perv^{\ULA}(A\times X/X)$. 

Let $K^0$ be the image of $h:K^1\to K$ in the abelian category $\Perv^{\ULA}(A\times X/X)$. By Fact \ref{ft:relperv} \ref{it:onDcb}, the functor $(-)|_{A_{\eta}}:\Perv(A\times X/X)\to \Perv(A_{\eta})$ is exact. Then $K^0|_{A_{\eta}}$ is the image of $h|_{A_{\eta}}:K^1|_{A_{\eta}}\to K|_{A_{\eta}}$ in $\Perv(A_{\eta})$. By Condition \ref{it:chinotKeta}, 	because $P^{\chi}(A_{\eta})$ is an abelian subcategory of $\Perv(A_{\eta})$,  the image of $h|_{A_{\eta}}:K^1|_{A_{\eta}}\to K|_{A_{\eta}}$ in $P^{\chi}(A_{\eta})$ is still $K^0|_{A_{\eta}}$. As the functor (\ref{eq:omegachi}) is exact, the image of $\omega_{\chi}(h|_{A_{\eta}}):\omega_{\chi}(K^1|_{A_{\eta}})\to \omega_{\chi}(K|_{A_{\eta}})$ is $\omega_{\chi}(K^0|_{A_{\eta}})$. 
	Combining (\ref{eq:global}), (\ref{eq:proj}), (\ref{eq:spectral}) with (\ref{eq:omegaK1}), one gets $\omega_{\chi}(K^0|_{A_{\eta}})=\omega_{\chi_{\eta}}(K|_{A_{\eta}})^{\Gamma_{k(\eta)}}$.\end{proof}
\subsection{Normality}
By \cite[Thm.~4.3]{javanpeykar2023monodromy}, for every character $\chi\in\cC(A)$, the geometric generic convolution group $G_{\omega_{\chi}}(K|_{A_{\bar{\eta}}})$ is a normal closed subgroup of the generic convolution group $G_{\omega_{\chi}}(K|_{A_{\eta}})$. Theorem \ref{thm:normal} shows that for uncountably many characters,  	the corresponding monodromy group is also a normal closed subgroup of the generic convolution group.

For every $\chi_{\ell'}\in \cC(A)_{\ell'}$ and every $\chi_{\ell}\in \cC(A)_{\ell}$,  set $\chi=\chi_{\ell'}\chi_{\ell}$.  
\begin{thm}\label{thm:normal}Let $K\in\Perv(A\times X/X)$ be semisimple $D_c^b(A\times X)$.
  Then for every $\chi_{\ell'}\in \cC(A)_{\ell'}$, there is a countable union $B=\cup_{i\ge 1}B_i$ of strict closed subsets of $\cC(A)_{\ell}$, such that for every  $\chi_{\ell}\in \cC(A)_{\ell}(\bar{\Q}_{\ell})\setminus B$, \begin{itemize}
 \item one has $K|_{A_{\eta}}\in P^{\chi}(A_{\eta})$;  
 \item  the algebraic group $G_{\omega_{\chi}}(K|_{A_{\eta}})$ is reductive;
 \item and $\Mon(K,\chi)$ is a \emph{normal} closed subgroup of $G_{\omega_{\chi}}(K|_{A_{\eta}})$.
 \end{itemize}\end{thm}
By Lemma \ref{lm:XnBaire}, when $\dim A>0$,  the set $\cC(A)_{\ell}(\bar{\Q}_{\ell})\setminus B$ is uncountable. Thus, Theorem \ref{thm:mainnormal} follows from Theorem \ref{thm:normal}. We sketch the proof of Theorem \ref{thm:normal}. By Theorem \ref{thm:introfixed}, for every representation $V$ of the Tannakian group $G(K|_{A_{\eta}})$ and every $\chi_{\ell'}\in \cC(A)_{\ell'}$,  there is a strict Zariski closed subset $B_V$ of the cotorus $\cC(A)_{\ell}$, such that for every $\chi_{\ell}\in (\cC(A)_{\ell}\setminus B_V)(\bar{\Q}_{\ell})$, the monodromy invariant $V^{\Mon(K,\chi)}$ is a $G(K|_{A_{\eta}})$-subrepresentation. Choose $B= \cup_VB_V(\bar{\Q}_{\ell})$. From Lemma \ref{lm:normalcriterion},  normality holds when $\chi_{\ell}\notin B$. 
\begin{proof}  Both $\Mon(K,\chi)$ and $G_{\omega_{\chi}}(K|_{A_{\eta}})$ depend only on the generic fiber of $p_X:A\times X\to X$. Therefore, shrinking $X$ to a nonempty open subset does not change  them. Thus, one may assume that $X$  is smooth. By \cite[Thm.~2.13, p.242]{SGA4.5}, one may assume further $K\in \Perv^{\ULA}(A\times X/X)$. By smoothness of $X$ and Lemma \ref{lm:genericfibersemisimple}, the object $K|_{A_{\eta}}$ of $\Perv(A_{\eta})$ is semisimple. 
From Lemma \ref{lm:Serrequotient} \ref{it:LongLiu}, $K|_{A_{\eta}}$ is also semisimple in $\bar{P}(A_{\eta})$. Therefore, a (hence every) Tannakian group of the neutral Tannakian category $\langle K|_{A_{\eta}}\rangle(\subset \bar{P}(A_{\eta}))$ is a \emph{reductive}, algebraic group over $\bar{\Q}_{\ell}$.
 Then by Lemma \ref{lm:repred}, there is a countable sequence of objects $\{\bar{K}_i\}_{i\ge1}$, such that every object of $\langle K|_{A_{\eta}}\rangle$ is isomorphic to some $\bar{K}_i$. To apply Theorem \ref{thm:introfixed}, we need semisimple objects of $D_c^b(A\times X)$.
	
	\begin{claim}\label{it:liftgeneric} For every object $N\in \langle K|_{A_{\eta}}\rangle$, there is  $L\in  \Perv^{\ULA}(A\times X/X)$ that is semisimple in $D_c^b(A\times X)$, such that $L|_{A_{\eta}}$ is isomorphic to $N$ in $\bar{P}(A_{\eta})$.\end{claim}
	
From Claim \ref{it:liftgeneric}, for every integer $i\ge1$, there is $K_i\in \Perv^{\ULA}(A\times X/X)$ that is semisimple in $D_c^b(A\times X)$ with $K_i|_{A_{\eta}}$ isomorphic to $\bar{K}_i$ in $\bar{P}(A_{\eta})$. From smoothness of $X$ and Theorem \ref{thm:introfixed}, there is a subobject $K_i^0\subset K_i$ in  $\Perv^{\ULA}(A\times X/X)$ and a strict Zariski closed subset $B_i\subset\cC(A)_{\ell}$, such that for every $\chi_{\ell}\in(\cC(A)_{\ell}\setminus B_i)(\bar{\Q}_{\ell})$, one has $K_i|_{A_{\eta}}\in P^{\chi}(A_{\eta})$ and \begin{equation}\label{eq:outBiinv}\omega_{\chi_{\eta}}(K_i|_{A_{\eta}})^{\Gamma_{k(\eta)}}=\omega_{\chi_{\eta}}(K_i^0|_{A_{\eta}}).\end{equation}

	Set $B:=\cup_{i\ge 1}B_i$.  For	every $\chi_{\ell}\in  \cC(A)_{\ell}(\bar{\Q}_{\ell})\setminus B$, one has $K|_{A_{\eta}}\in P^{\chi}(A_{\eta})$. For every $i\ge 1$, by $\chi_{\ell}\notin B_i(\bar{\Q}_{\ell})$ and \eqref{eq:outBiinv}, the subspace $\omega_{\chi_{\eta}}(K_i|_{A_{\eta}})^{\Mon(K,\chi)}$ is $G_{\omega_{\chi}}(K|_{A_{\eta}})$-stable. By Lemmas \ref{lm:normalcriterion} and \ref{lm:monreductive}, the  subgroup $\Mon(K,\chi)$  of $G_{\omega_{\chi}}(K|_{A_{\eta}})$ is  \emph{normal}.
\end{proof}
	\begin{proof}[Proof of Claim \ref{it:liftgeneric}]	From Lemma \ref{lm:repred},  the object $N\in \bar{P}(A_{\eta})$ is semisimple. There is an integer $n\ge0$ such that $N$ is a subquotient of $(K|_{A_{\eta}}\oplus K|_{A_{\eta}}^{\vee})^{*n}$ in $\bar{P}(A_{\eta})$. 
		
		We ``globalize" the fiberwise convolution functors as follows.	Define a bifunctor \begin{equation}\label{eq:rel*}\begin{aligned}*_X:& D_c^b(A\times X)\times D_c^b(A\times X)\to D_c^b(A\times X),\\
				&(-,+)\mapsto R(m\times \Id_X)_*(p_{13}^*-\otimes^Lp_{23}^*+),\end{aligned}
		\end{equation} where $p_{ij}$ are the projections on $A\times A\times X$. By the proper base change theorem, for every $x\in X(k)$, one has $(-*_X+)|_{A_x}\iso (-|_{A_x})*(+|_{A_x})$ as bifunctors $D_c^b(A\times X)\times D_c^b(A\times X)\to D_c^b(A_x)$. Therefore, one has $(-*_X+)|_{A_{\eta}}\iso(-|_{A_{\eta}})*(+|_{A_{\eta}})$ as bifunctors  $D_c^b(A\times X)\times D_c^b(A\times X)\to D_c^b(A_{\eta})$.

We prove that	the bifunctor (\ref{eq:rel*}) restricts to a bifunctor $D^{\ULA}(A\times X/X)\times D^{\ULA}(A\times X/X)\to D^{\ULA}(A\times X/X)$. By Fact \ref{ft:ULA} \ref{it:externaltensorULA}, for any $K',K''\in D^{\ULA}(A\times X/X)$,   one has \[p_{13}^*K'\otimes^Lp_{23}^*K''\in D^{\ULA}(A\times A\times X/X).\]  By Fact \ref{ft:ULA} \ref{it:properULA}, one gets $K'*_XK''\in D^{\ULA}(A\times X/X)$.
		
		Set $K^{\vee}:=([-1]_A\times \Id_X)^*\D_{A\times X/X}K$. By Fact \ref{ft:relperv} \ref{it:relVerdier}, one has $K^{\vee}\in \Perv^{\ULA}(A\times X/X)$ and $(K^{\vee})|_{A_{\eta}}=(K|_{A_{\eta}})^{\vee}$. Then \[(K\oplus K^{\vee})^{*_Xn})\in D^{\ULA}(A\times X/X).\] Set $M:={}^{p/X}\cH^0((K\oplus K^{\vee})^{*_Xn})\in\Perv^{\ULA}(A\times X/X)$. Then $M|_{A_{\eta}}={}^p\cH^0([K|_{A_{\eta}}\oplus (K|_{A_{\eta}})^{\vee}]^{*n})$ in $\Perv(A_{\eta})$. By Lemma \ref{lm:Serrequotient} \ref{it:liftsssq}, there is a semisimple subquotient $L'$ of $M|_{A_{\eta}}$ in $\Perv(A_{\eta})$, whose image in $\bar{P}(A_{\eta})$ is $N$. By smoothness of $X$ and Fact \ref{ft:relULAingeneric}, there is a semisimple subquotient $L$ of $M$ in $\Perv^{\ULA}(A\times X/X)$ with $L|_{A_{\eta}}=L'$. By smoothness of $X$ and Lemma \ref{lm:Scholze} \ref{it:ScholzeSerre}, the object $L[\dim X]$ is  semisimple in $\Perv(A\times X)$.  Then $L$ is semisimple in $D_c^b(A\times X)$. \end{proof}
	For a category $\cC$, let $\cC/\sim$ be the class of isomorphism classes of objects in $\cC$.\begin{lm}\label{lm:repred}
Let $(\cC,\otimes)$ be a neutral Tannakian category over an algebraically closed field $k$ of characteristic $0$ with a fiber functor $\omega:\cC\to \Vec_k$. Assume that $\Aut^{\otimes}(\cC,\omega)$ is a reductive, algebraic group over $k$. Then the underlying abelian category is semisimple, and $\cC/\sim$ is countable. 
\end{lm}
\begin{proof}Set $G=\Aut^{\otimes}(\cC,\omega)$. Let $\Rep(G)$ be the category of $k$-rational representations of $G$.  Then  $\cC$ is equivalent to $\Rep(G)$. By \cite[Cor.~22.43]{milne2017algebraic}, 
because $k$ has characteristic $0$, the abelian category $\Rep(G)$ is semisimple. By \cite[Thm.~2.16]{achar2020representation}, as $k$ is algebraically closed, there is an  at most countable set $X^+$ and for every $\lambda\in X^+$, a unital  $k$-algebra $\mathscr{A}^{\lambda}$ with the following property: The set $\mathrm{Irr}(G)$ of isomorphism classes of simple objects of $\Rep(G)$ is in bijection with the set of pairs $(\lambda,E)$, where $\lambda\in X^+$ and $E$ is an isomorphism class of simple left $\mathscr{A}^{\lambda}$-modules. From \cite[Lem.~2.19]{achar2020representation}, for every $\lambda\in X^+$,  the algebra $\mathscr{A}^{\lambda}$ is semisimple. Then by \cite[XVII, Thm.~4.3, Cor.~4.5]{Lang2002algebra}, the set of isomorphism classes of simple left $\mathscr{A}^{\lambda}$-modules is finite. Therefore, $\mathrm{Irr}(G)$ is at most countable. Consequently, $\Rep(G)/\sim$ is countable.	
\end{proof}
\begin{lm}\label{lm:Serrequotient}
	Let $\cA$ be an abelian category. Let $\cB\subset\cA$ be a Serre  subcategory. Consider the quotient functor $F:\cA\to \cA/\cB$. 
	\begin{enumerate}\item\label{it:LongLiu}  Let $X\in\cA$.  Let $i:Y\to F(X)$ be a monomorphism  in $\cA/\cB$. Then there is a monomorphism $j:Z\to X$ in $\cA$ and an isomorphism $u:Y\to F(Z)$ in $\cA/\cB$ fitting into a commutative diagram in $\cA/\cB$ \begin{center}
			\begin{tikzcd}
				& F(Z) \arrow[rd, "F(j)"] &   \\
				Y \arrow[ru, "u"] \arrow[rr, "i", hook] &                      & F(X).
			\end{tikzcd}
		\end{center} Dually, up to isomorphism every quotient in $\cA/\cB$ lifts to a quotient in $\cA$. In particular, if $X\in \cA$ is a simple object, then $F(X)$ is either simple or zero in $\cA/\cB$.
		\item\label{it:liftsimple} Let $V\in \cA$ be a Noetherian and Artinian object. If $F(V)$ is simple in $\cA/\cB$, then there is a simple subquotient $W$ of $V$ in $\cA$ such that $F(W)$ is naturally isomorphic to $F(V)$ in $\cA/\cB$.
		\item\label{it:liftsssq} Assume that $\cA$ is Noetherian and Artinian. Let $X\in\cA$.  Let $Y$ be a simple subquotient of $F(X)$ in $\cA/\cB$. Then there is a simple subquotient $W$ of $X$, with $F(W)$ naturally isomorphic to $Y$ in $\cA/\cB$.
	\end{enumerate}
\end{lm}\begin{proof}
\hfill	\begin{enumerate}
		\item By the construction in the proof of \cite[\href{https://stacks.math.columbia.edu/tag/02MS}{Tag 02MS}]{stacks-project} and the right calculus of fractions in \cite[\href{https://stacks.math.columbia.edu/tag/04VB}{Tag 04VB}]{stacks-project}, there is a diagram \begin{center}
			\begin{tikzcd}
				& M \arrow[rd, "g"] \arrow[ld, "f"] &   \\
				Y &                                   & X
			\end{tikzcd}
		\end{center} in $\cA$, such that $F(f)$ is an isomorphism and $F(g)=i\circ F(f)$ in $\cA/\cB$. Therefore, $F(g)$ is a monomorphism. Since $F$ is exact, one has $F(\ker(g))=\ker(F(g))=0$, so $\ker(g)\in \cB$. Let $q:M\to M/\ker(g)$ be the epimorphism in $\cA$, and let $j:M/\ker(g)\to X$ be the monomorphism in $\cA$ induced by $g$. Then $F(q)$ is an isomorphism in $\cA/\cB$. Set $u:Y\to F(M/\ker(g))$ to be the morphism $F(q)\circ F(f)^{-1}$ in $\cA/\cB$. Then $u$ is an isomorphism with the stated property.
		\item Let	$\cP$ be the family of subobjects $V'$ of $V$ in $\cA$ with $V/V'\in \cB$. Then $\cP$ is nonempty since $V\in P$. As $V$ is Artinian in $\cA$,  there is a minimal object $U\in \cP$. Moreover, the morphism $F(U)\to F(V)$ is an isomorphism in $\cA/\cB$.
		Let	$\cQ$ be the family of subobjects of $U\in \cA$ lying in $\cB$. Then $\cQ$ is nonempty since $0\in \cQ$. As $V$ is Noetherian in $\cA$, so is $U$. Thus, $\cQ$ has a maximal object $U_0$. Then $W:=U/U_0$ is a subquotient of $V\in\cA$ and the morphism $F(U)\to F(W)$ is an isomorphism in $\cA/\cB$. In particular, $W\neq0$ in $\cA$. 
		
		It remains to prove that	$W$ is simple in $\cA$. For this, let $U'\to W$ be a  subobject in $\cA$. Then there is a subobject $U''$ of $U$ in $\cA$ containing $U_0$ with  $U''/U_0=U'$. As $F(U'')$ is a subobject of a simple object $F(U)$ in $\cA/\cB$,  either the morphism $F(U'')\to F(U)$ is an isomorphism or $F(U'')=0$. If $F(U'')=0$, then $U''\in \cB$ and $U''\in\cQ$. Since $U_0$ is maximal in $\cQ$, one has $U_0=U''$, so $U'=0$. If $F(U'')\to F(U)$ is an isomorphism, then $U/U''\in \cB$. Since the sequence \[0\to U/U''\to V/U''\to V/U\to0\] is exact in $\cA$, and $\cB$ is closed under extensions,  one gets $V/U''\in \cB$ and $U''\in \cP$. Since $U$ is minimal in $\cP$, one has $U''=U$. The morphism $U'\to W$ is thus an isomorphism in $\cA$. 
		\item By Part \ref{it:LongLiu}, there is a subquotient $Z$ of $X$ in $\cA$ with $F(Z)$ naturally isomorphic to $Y$. Then $F(Z)$ is simple in $\cA/\cB$. By assumption, $Z$ is  Noetherian and Artinian in $\cA$. Thus from Part \ref{it:liftsimple}, there is a simple subquotient $W$ of $Z$ in $\cA$ with $F(W)$ naturally isomorphic to $F(Z)$ and to $Y$ in $\cA/\cB$. \end{enumerate}
\end{proof}

	\subsection*{Acknowledgments}
I am grateful to my advisor, Prof.~Anna Cadoret, for  listening to my oral reports on this work and pointing out a  flaw therein. Lemma \ref{lm:Bairecomp} \ref{it:hypersurface} and the proof of Lemma \ref{lm:repred} are due to her. I also benefited from her various constructive advice and multiple careful reading. Prof.~Peter Scholze kindly answered my question, and provided the proof of  Lemma \ref{lm:Scholze}. I appreciate the patience and detailed replies of Professors Owen Barrett, Marco Maculan and Will Sawin  to my questions on their respective work. I thank Emiliano Ambrosi for his hospitality during my visit to the Université de Strasbourg. During the preparation, I also received the help of Chenyu Bai, Arnaud Eteve, Arnab Kundu, Junbang Liu, Long Liu, Kai Mao, Keyao Peng, Mingchen Xia, Junsheng Zhang and Xiaoxiang Zhou. I thank Gabriel Ribeiro for noting  a mistake in a previous version. Hui Zhang helped me out multiple times with his admirable knowledge, especially in algebraic geometry. All remaining errors are mine.
\printbibliography
\end{document}